\documentclass[reqno]{amsart}

\usepackage{amsmath}
\usepackage{amssymb}
\usepackage{amsthm}
\usepackage{color}
\usepackage{verbatim}
\usepackage{tabularx}
\usepackage{ulem} 
\usepackage{hyperref}

\usepackage{geometry}
\usepackage{mathrsfs}

\theoremstyle{theorem}
\newtheorem{theorem}{Theorem}

\newtheorem{corollary}[theorem]{Corollary}

\newtheorem{lemma}[theorem]{Lemma}

\newtheorem{proposition}[theorem]{Proposition}

\theoremstyle{definition}
\newtheorem{remark}[theorem]{Remark}
\newtheorem*{remark*}{Remark}

\newcommand{\N}{\mathbb{N}}
\newcommand{\R}{\mathbb{R}}

\newcommand{\metric}{\langle \, , \, \rangle}

\newcommand{\eps}{\varepsilon}

\newcommand{\di}{\mathrm{d}}

\newcommand{\loc}{\mathrm{loc}}

\newcommand{\RR}{\mathbb{R}}

\newcommand{\esssup}{\mathrm{ess\,sup}}

\DeclareMathOperator{\dist}{dist}

\renewcommand{\div}{\mathrm{div}}

\makeatletter
\newcommand*\owedge{\mathpalette\@owedge\relax}
\newcommand*\@owedge[1]{
	\mathbin{
		\ooalign{
			$#1\m@th\bigcirc$\cr
			\hidewidth$#1\m@th\wedge$\hidewidth\cr
		}
	}
}
\makeatother

\title[Growth of subsolutions of quasilinear equations]{Growth of subsolutions of $\Delta_p u = V|u|^{p-2}u$ and of a general class of quasilinear equations}

\author{Luis Al\'ias}
\address{Departamento de Matemáticas, Universidad de Murcia, Campus de Espinardo, E-30100 Espinardo, Murcia, Spain}
\email{ljalias@um.es}

\author{Giulio Colombo}
\address{Dipartimento di Matematica ``F. Enriques'', Universit\`a degli Studi di Milano, Via Saldini 50, I-20133 Milano, Italy.}
\email{giulio.colombo@unimi.it}

\author{Marco Rigoli}
\address{Dipartimento di Matematica ``F. Enriques'', Universit\`a degli Studi di Milano, Via Saldini 50, I-20133 Milano, Italy.}
\email{marco.rigoli@unimi.it}

\begin{document}

\maketitle

\begin{abstract}
	In this paper we prove some integral estimates on the minimal growth of the positive part $u_+$ of subsolutions of quasilinear equations
	\[
		\div A(x,u,\nabla u) = V|u|^{p-2}u
	\]
	on complete Riemannian manifolds $M$, in the non-trivial case $u_+\not\equiv 0$. Here $A$ satisfies the structural assumption $|A(x,u,\nabla u)|^{p/(p-1)} \leq k \langle A(x,u,\nabla u),\nabla u\rangle$ for some constant $k>0$ and for $p>1$ the same exponent appearing on the RHS of the equation, and $V$ is a continuous positive function, possibly decaying at a controlled rate at infinity. We underline that the equation may be degenerate and that our arguments do not require any geometric assumption on $M$ beyond completeness of the metric. From these results we also deduce a Liouville-type theorem for sufficiently slowly growing solutions.
\end{abstract}

\bigskip



\bigskip

\section{Introduction}

In the recent paper \cite{cmr22} (Lemma 8), the following theorem was established. Let $M$ be a complete Riemannian manifold (without boundary), $\lambda>0$ a constant and $u\in C^2(M)$. If the superlevel set $\Omega_+ := \{x\in M : u(x)>0\}$ is not empty and $u$ satisfies
\begin{equation} \label{in0}
	\Delta u \geq \lambda u \qquad \text{on } \, \Omega_+
\end{equation}
then for any fixed point $x_0\in M$ we have
\begin{equation} \label{ex0}
	\liminf_{R\to+\infty} \frac{1}{R} \log\int_{B_R(x_0)} u_+^2 > 0
\end{equation}
where $u_+ := \max\{u,0\}$ is the positive part of $u$ and $B_R(x_0)$ is the geodesic ball of radius $R$ centered at $x_0$. Indeed, inspection of the proof also shows that there exists a constant $C(\lambda)$, not depending on $M$ or $u$, such that
\begin{equation} \label{ex1}
	\liminf_{R\to+\infty} \frac{1}{R} \log\int_{B_R(x_0)} u_+^2 \geq C(\lambda) > 0
\end{equation}
and that the optimal value for $C(\lambda)$ is not smaller than $\frac{\log 2}{4}\sqrt{\lambda}$. This can be regarded as a sort of ``gap'' theorem for subsolutions of $\Delta u = \lambda u$: if $u\in C^2(M)$ satisfies
\[
	\Delta u \geq \lambda u \qquad \text{on } \, M
\]
then either $u\leq 0$ or the positive part of $u$ has to be sufficiently large in an integral sense (that is, its $L^2$ norm on $B_R(x_0)$ must grow at least exponentially with respect to $R$). In fact, the result from \cite{cmr22} is more general and also covers the case of weighted Laplacians and locally Lipschitz weak solutions of \eqref{in0}.

In this paper we generalize the above theorem by considering differential inequalities for a wider class of (possibly degenerate) quasilinear elliptic operators in divergence form, including the $p$-Laplace operator
\[
	\Delta_p u := \div(|\nabla u|^{p-2}\nabla u) \, , \qquad 1 < p < +\infty \, ,
\]
and also replacing the constant $\lambda$ by a positive continuous function $V$ possibly decaying at infinity at a controlled rate, namely, not faster than a negative power $r(x)^{-\mu}$, $\mu>0$, of the distance $r(x)=\dist(x,o)$ from some fixed point $o\in M$. More precisely, for a given pair of parameters $\lambda>0$ and $\mu\geq 0$ we shall assume that
\begin{equation} \tag{$V_{\lambda,\mu}$} \label{Vlm}
	\begin{array} {r@{\;}c@{\;}ll}
		V & \geq & \lambda & \qquad \text{if } \, \mu = 0 \\[0.2cm]
		\displaystyle\liminf_{x\to\infty} \, [\dist(x,o)^\mu V(x)] & \geq & \lambda \quad \text{for some } \, o\in M & \qquad \text{if } \, \mu > 0 \, .
	\end{array}
\end{equation}
These conditions are clearly satisfied, for instance, if
\[
	V(x) \geq \frac{\lambda}{1 + \dist(x,o)^\mu} \qquad \text{on } \, M \, .
\]
Also, in case $\mu>0$ the triangle inequality implies that the validity of \eqref{Vlm} does not depend on the choice of the reference base point $o\in M$.

To give an example of our main result, we state it in the model case of the $p$-Laplace operator. To do so, we have to precise some terminology. For a function $u\in W^{1,p}_\loc(M)$, we denote by $\Omega_+ := \{x\in M : u(x)>0\}$ its positivity set and for a given measurable function $V\geq 0$ we say that $u$ satisfies
\begin{equation} \label{eq2-in}
	\Delta_p u \geq V u^{p-1} \qquad \text{weakly on } \, \Omega_+
\end{equation}
if
\[
	-\int_M \langle |\nabla u|^{p-2}\nabla u,\nabla\varphi\rangle \geq \int_M V u^{p-1} \varphi \qquad \forall \, \varphi \in D^+(\Omega_+)
\]
where
\begin{align*}
	D^+(\Omega_+) := \{ \varphi\in W^{1,p}_c(M) : \; & \varphi\geq 0 \text{ on } M , \\
	& \varphi=0 \text{ and } \nabla\varphi=0 \text{ a.e.~on } M\setminus\Omega_+ \} \, .
\end{align*}
(Note that if $|\Omega_+|>0$ then the space $D^+(\Omega_+)$ of test functions is non-trivial because it contains at least elements of the form $\varphi = u_+\psi$, with $0\leq\psi\in C^\infty_c(M)$, so \eqref{eq2-in} is a meaningful condition.) In particular, \eqref{eq2-in} is always satisfied if
\begin{equation} \label{dVu}
	\Delta_p u \geq V|u|^{p-2}u \qquad \text{weakly on } \, M
\end{equation}
or even if
\begin{equation} \label{dVu+}
	\Delta_p u_+ \geq Vu_+^{p-1} \qquad \text{weakly on } \, M
\end{equation}
since $\nabla u_+ = \mathbf{1}_{\Omega_+} \nabla u$ almost everywhere on $M$. Note that \eqref{dVu+} is a weaker condition than \eqref{dVu}, as follows from work of Le, \cite{le98}.

\begin{theorem} \label{pL-in}
	Let $M$ be a complete Riemannian manifold, $p\in(1,+\infty)$, $\mu\in[0,p]$, $\lambda>0$, and $V : M \to (0,+\infty)$ a continuous function satisfying \eqref{Vlm}.
	
	Let $u\in W^{1,p}_\loc(M)$. If $\Omega_+ := \{x\in M : u(x)>0\}$ is of positive measure and
	\[
		\Delta_p u \geq V u^{p-1} \qquad \text{weakly on } \, \Omega_+
	\]
	then for any $x_0\in M$ and $q\in(p-1,+\infty)$ we have
	\begin{alignat}{2}
		\label{iwp-0}
		\liminf_{R\to+\infty} \, \frac{1}{R^{1-\frac{\mu}{p}}}
		\log\displaystyle\int_{B_R(x_0)} u_+^q & \geq \frac{C_0}{1-\frac{\mu}{p}} > 0 && \qquad \text{if } \, \mu \in [0,p) \\
		\label{iwp-1}
		\liminf_{R\to+\infty} \, \frac{1}{\log R}
		\log\displaystyle\int_{B_R(x_0)} u_+^q & \geq C_1 > p && \qquad \text{if } \, \mu = p
	\end{alignat}
	where $C_0$ and $C_1$ are explicitely given by
	\begin{equation} \label{C01}
		C_0 = \frac{p(q-p+1)^{1/p'}}{(p-1)^{1/p'}} \lambda^{1/p} \, , \qquad C_1^{1/p}(C_1-p)^{1/p'} = C_0
	\end{equation}
	where $p'=\frac{p}{p-1}$ is the exponent conjugate to $p$. Moreover, in case $\mu=p$ we have
	\begin{equation} \label{iwp-1b}
		\lim_{R\to+\infty} \frac{1}{\log R} \log\int_{B_R(x_0)} u_+^q \geq C_0 + p
	\end{equation}
	whenever the limit on the LHS exists.
\end{theorem}

\begin{remark}
	Note that the value $C_1>p$ determined by \eqref{C01} satisfies $C_1<C_0+p$, hence \eqref{iwp-1b} gives a stronger estimate than \eqref{iwp-1} when its LHS is well defined.
\end{remark}

The constants appearing in \eqref{iwp-0} and \eqref{iwp-1b} are sharp, that is, for each combination of values of $p$, $\mu$, $\lambda$ and $q$ it is possible to find $M$ and $u$ for which the equality in \eqref{iwp-0} or \eqref{iwp-1b} is attained. This is shown by explicit examples described at the end of Section \ref{sec3}. We don't know whether the value of $C_1>p$ in \eqref{C01} is sharp or not for the validity of \eqref{iwp-1}. It seems worth to underline that the case $p=2$, $q=2$, $\mu=0$ in the above theorem implies that the optimal value for $C(\lambda)$ in \eqref{ex1} is $C(\lambda)=2\sqrt{\lambda}$.

\begin{theorem} \label{2L-in}
	Let $M$ be a complete Riemannian manifold, $\mu\in[0,2]$, $\lambda>0$ and $V : M \to (0,+\infty)$ a continuous function satisfying \eqref{Vlm}.
	
	Let $u\in W^{1,2}_\loc(M)$. If $\Omega_+ := \{x\in M : u(x) > 0\}$ is of positive measure and
	\[
		\Delta u \geq V u \qquad \text{weakly on } \, \Omega_+
	\]
	then for any $x_0\in M$ and $q\in(1,+\infty)$ we have
	\begin{alignat*}{2}
		\liminf_{R\to+\infty} \, \frac{1}{R^{1-\frac{\mu}{2}}}
		\log\displaystyle\int_{B_R(x_0)} u_+^q & \geq \frac{2\sqrt{q-1}\sqrt{\lambda}}{1-\frac{\mu}{2}} && \qquad \text{if } \, \mu \in [0,2) \\
		\liminf_{R\to+\infty} \, \frac{1}{\log R}
		\log\displaystyle\int_{B_R(x_0)} u_+^q & \geq 1+\sqrt{1+4(q-1)\lambda} > 2 && \qquad \text{if } \, \mu = 2
	\end{alignat*}
	and in case $\mu=2$
	\[
		\lim_{R\to+\infty} \frac{1}{\log R} \log\int_{B_R(x_0)} u_+^q \geq 2(1+\sqrt{q-1}\sqrt{\lambda})
	\]
	provided the limit exists.
\end{theorem}

In full generality, in our main theorem we deal with differential inequalities involving quasilinear differential operators $L$ formally defined by
\begin{equation} \label{Lu_def}
	Lu := \div(A(x,u,\nabla u))
\end{equation}
where $A : \R\times TM \to TM$ is a continuous function (or, more generally, a Carath\'eodory-type function as specified in Section \ref{sec2}) satisfying
\begin{equation} \label{A_in}
	\langle A(x,s,\xi),\xi\rangle \geq 0 \qquad \text{and} \qquad |A(x,s,\xi)| \leq k\langle A(x,s,\xi),\xi\rangle^{\frac{p}{p-1}}
\end{equation}
for all $x\in M$, $s\in\R$, $\xi\in T_x M$ with some constant $k>0$. If these conditions are satisfied, we say that the differential operator $L$ defined by \eqref{Lu_def} is weakly-$p$-coercive with coercivity constant $k$. The $p$-Laplace operator falls in this class since it can be expressed as in \eqref{Lu_def} for the choice $A(x,s,\xi) = |\xi|^{p-2}\xi$, which fulfills \eqref{A_in} with $k=1$. In analogy with what we did above, we say that a function $u\in W^{1,p}_\loc(M)$ satisfies
\[
	Lu \geq V u^{p-1} \qquad \text{weakly on } \, \Omega_+ := \{u>0\}
\]
if
\[
	-\int_M \langle A(x,u,\nabla u),\nabla\varphi\rangle \geq \int_M V u^{p-1} \varphi \qquad \forall \, \varphi \in D^+(\Omega_+) \, .
\]

\begin{theorem} \label{main_in}
	Let $M$ be a complete Riemannian manifold, $p\in(1,+\infty)$, $\mu\in[0,p]$ and $\lambda>0$. Let $L$ be a weakly-$p$-coercive operator as in \eqref{Lu_def} with coercivity constant $k>0$ and $V : M \to (0,+\infty)$ a continuous function satisfying \eqref{Vlm}.
	
	Let $u\in W^{1,p}_\loc(M)$. If $\Omega_+ := \{ x\in M : u(x) > 0 \}$ is of positive measure and
	\[
		L u \geq V u^{p-1} \qquad \text{weakly on } \, \Omega_+
	\]
	then for any $x_0\in M$ and $q\in(p-1,+\infty)$ we have
	\begin{alignat}{2}
		\label{iwp-0L}
		\liminf_{R\to+\infty} \, \frac{1}{R^{1-\frac{\mu}{p}}}
		\log\displaystyle\int_{B_R(x_0)} u_+^q & \geq \frac{C_0}{1-\frac{\mu}{p}} && \qquad \text{if } \, \mu \in [0,p) \\
		\label{iwp-1L}
		\liminf_{R\to+\infty} \, \frac{1}{\log R}
		\log\displaystyle\int_{B_R(x_0)} u_+^q & \geq C_1 && \qquad \text{if } \, \mu = p
	\end{alignat}
	where $C_0>0$ and $C_1>p$ are determined by
	\[
		C_0 = \frac{p(q-p+1)^{1/p'}}{(p-1)^{1/p'}} \frac{\lambda^{1/p}}{k} \, , \qquad C_1^{1/p}(C_1-p)^{1/p'} = C_0
	\]
	with $p'=\frac{p}{p-1}$. Moreover, in case $\mu=p$ we have
	\begin{equation} \label{iwp-1bL}
		\lim_{R\to+\infty} \frac{1}{\log R} \log\int_{B_R(x_0)} u_+^q \geq C_0 + p
	\end{equation}
	whenever the limit exists.
\end{theorem}

We point out that the RHS's of \eqref{iwp-1L} and \eqref{iwp-1bL} both converge to $p$ from above as $\lambda\to0^+$. Hence, if $u\in W^{1,p}_\loc(M)$ satisfies
\[
	Lu \geq Vu^{p-1} \qquad \text{weakly on } \, \Omega_+ = \{u>0\}
\]
with $|\Omega_+|\neq 0$ and $V$ a continuous positive function decaying to $0$ faster than $r(x)^{-p}$ as $x\to\infty$, then on arbitrary manifolds we couldn't expect the possible validity of an estimate stronger than
\[
	\liminf_{R\to+\infty} \frac{1}{\log R} \log\int_{B_R} u_+^q \geq p \, .
\]
In fact, we are able to prove a weaker growth estimate (with $\liminf$ replaced by $\limsup$) holds more generally for any $u\in W^{1,p}_\loc(M)$ satisfying
\begin{equation} \label{Luf_in}
	Lu\geq f \qquad \text{weakly on } \, \Omega_+
\end{equation}
for some measurable function $f:M\to[0,+\infty]$ such that $f>0$ on a set $E\subseteq\Omega_+$ of positive measure. Of course, by \eqref{Luf_in} we mean that
\begin{equation} \label{Luf_ind}
	-\int_M \langle A(x,u,\nabla u),\nabla\varphi\rangle \geq \int_M f\varphi \qquad \forall \, \varphi\in D^+(\Omega_+) \, .
\end{equation}
Note that if \eqref{Luf_in} holds with $f$ as above then there exists $\varphi\in D^+(\Omega_+)$ for which the LHS of \eqref{Luf_ind} is strictly positive (this follows by considering a test function of the form $\varphi = u_+\psi$ for some $0\leq\psi\in C^\infty_c(M)$ strictly positive on a portion of $E$ of positive measure), and then it must also be $A(x,u,\nabla u)\neq 0$ on a subset $E_0\subseteq\Omega_+$ of positive measure.

\begin{theorem} \label{thm-in4}
	Let $M$ be a complete Riemannian manifold, $p\in(1,+\infty)$, $L$ a weakly-$p$-coercive operator as in \eqref{Lu_def} and $u\in W^{1,p}_\loc(M)$ such that $\Omega_+ := \{x\in M : u(x)>0\}$ has positive measure. If $u$ satisfies
	\[
		L u \geq 0 \qquad \text{weakly on } \, \Omega_+
	\]
	and further
	\begin{equation} \label{An0}
		A(x,u,\nabla u) \neq 0 \qquad \text{on a set $E_0\subseteq\Omega_+$ of positive measure}
	\end{equation}
	then for any $q\in(p-1,+\infty)$
	\begin{equation} \label{eq1}
		\limsup_{R\to+\infty} \frac{1}{R^p} \int_{B_R} u_+^q = +\infty \, .
	\end{equation}
	In particular,
	\[
		\limsup_{R\to+\infty} \frac{1}{\log R} \log\int_{B_R} u_+^q \geq p \, .
	\]
\end{theorem}

As said, \eqref{An0} holds if $u$ satisfies \eqref{Luf_in} for some measurable $f:M\to[0,+\infty]$ with $f$ not a.e.~vanishing on $\Omega_+$. Alternatively, \eqref{An0} is satisfied also when $u$ is not constant on $M$ and positive somewhere (so that $|\Omega_+|>0$) and $A$ obeys the following mild non-degeneracy condition:
\begin{equation} \label{And-in}
	A(x,s,\xi) = 0 \qquad \text{only if} \qquad \xi = 0 \, .
\end{equation}
Theorem \ref{thm-in4} is a consequence of the next Theorem \ref{thm-in5}, proved in the last part of the paper where we extend some arguments from \cite{rs} to general weakly-$p$-coercive operators $L$ of the form \eqref{Lu_def}.

\begin{theorem} \label{thm-in5}
	Let $M$ be a complete, non-compact Riemannian manifold, $p\in(1,+\infty)$, $L$ a weakly-$p$-coercive operator as in \eqref{Lu_def} and $u\in W^{1,p}_\loc(M)$. If $\{u>0\}$ has positive measure, $u$ satisfies
	\begin{equation} \label{Lu0}
		L u \geq 0 \qquad \text{weakly on } \, \{u>0\}
	\end{equation}
	and for some $q>p-1$ it holds
	\begin{equation} \label{L1-in}
		\lim_{R\to+\infty} \int_r^R \left( \int_{\partial B_s} u_+^q \right)^{-\frac{1}{p-1}} \di s = +\infty \qquad \forall \, r > 0 \, ,
	\end{equation}
	then $A(x,u,\nabla u) = 0$ almost everywhere on $\{u>0\}$. In particular, if the structural condition \eqref{And-in} holds, then $u$ is constant on $M$.
\end{theorem}

We remark that condition \eqref{L1-in} amounts to saying that the function $\varphi : (0,+\infty) \to [0,+\infty]$ given by
\[
	\varphi(s) = \left( \int_{\partial B_s} (u-s_0)_+^q \right)^{-\frac{1}{p-1}} \qquad \forall \, s > 0
\]
is not in $L^1((r,+\infty))$ for any $r>0$. In fact, as proved in Lemma \ref{bd_int>0} below, in the assumptions of Theorem \ref{thm-in5} there exists $r_0\geq0$ such that $\varphi$ is finite a.e.~on $(r_0,+\infty)$ and $\varphi \in L^1((r,R))$ for any $r_0<r<R<+\infty$, so that \eqref{L1-in} is satisfied if and only if $\varphi$ is not integrable in a neighbourhood of $+\infty$. Note that in general $\varphi$ may be integrable at $+\infty$ and still satisfy $\varphi=+\infty$ on $(0,r_0)$ for some $r_0>0$. For instance, for fixed $n\in\N$ and $p>n$, the function
\[
	u(x) := |x|^{\frac{p-n}{p-1}} - 1 \qquad \text{on } \, \R^n
\]
satisfies $\Delta_p u = 0$ on $\Omega_+ = \R^n\setminus\overline{B_1}$, and for any $q>p-1$
\[
	\varphi(s) = \begin{cases}
		+\infty & \quad \text{for } \, 0 < s \leq 1 \\
		[C s^{n-1}(s^{q\frac{p-n}{p-1}}-1)]^{-\frac{1}{p-1}} & \quad \text{for } \, s > 1
	\end{cases}
\]
(with $C=|\partial B_1|$) is integrable at $+\infty$: indeed,
\[
	\varphi(s) \sim C^{-\frac{1}{p-1}} s^{-\frac{(n-1)(p-1)+q(p-n)}{(p-1)^2}} \qquad \text{as } \, s\to+\infty
\]
and (under the assumption $p>n$) we have $-\frac{(n-1)(p-1)+q(p-n)}{(p-1)^2}<-1$ if and only if $q>p-1$. This shows that the clause ``$\forall\,r>0$'' in \eqref{L1-in} cannot in general be replaced by ``for some $r>0$''.

Note that \eqref{L1-in} is a condition about the growth of the integral of $u_+^q$ on geodesic spheres $\partial B_s$. This can be related to the growth of the integral of $u_+^q$ on balls $B_s$. More precisely, \eqref{L1-in} is implied (see Proposition 1.3 in \cite{rs}) by the stronger condition
\[
	\lim_{R\to+\infty} \int_r^R \left( \frac{s}{\int_{B_s} u_+^q} \right)^{\frac{1}{p-1}} \di s = +\infty \qquad \forall \, r > 0
\]
which in turn is satisfied, for instance, when
\[
	\int_{B_R} u_+^q = O(R^p) \qquad \text{as } \, R\to+\infty \, .
\]
Since this last condition is exactly the negation of condition \eqref{eq1} above, Theorem \ref{thm-in4} follows at once from Theorem \ref{thm-in5}.

As hinted at the beginning of this Introduction, our main Theorem \ref{main_in} can be also interpreted as a ``gap'' theorem for functions $u\in W^{1,p}_\loc(M)$ satisfying
\[
	Lu \geq V|u|^{p-2}u \qquad \text{on } \, M \, .
\]
Namely, if $u$ satisfies the above differential inequality, then either $u\leq 0$ a.e.~on $M$ or the positive part of $u$ must grow sufficiently fast. As an easy consequence we have the following Liouville-type result (for its proof it is enough to apply Theorem \ref{main_in} to both $u$ and $-u$). For the sake of simplicity, we only state it in case $V$ is a positive constant, but the interested reader can immediately generalize it to the case where $V$ is a function satisfying \eqref{Vlm} for some $\lambda>0$ and $\mu\in[0,p]$.

\begin{theorem} \label{thm-in7}
	Let $M$ be a complete Riemannian manifold, $p\in(1,+\infty)$, $\lambda>0$ and $L$ a weakly-$p$-coercive operator as in \eqref{Lu_def} with coercivity constant $k>0$. Let $u\in W^{1,p}_\loc(M)$ satisfy
	\[
		Lu = \lambda |u|^{p-2}u \qquad \text{on } \, M \, .
	\]
	If for some $x_0\in M$ and $q\in(p-1,+\infty)$
	\[
		\int_{B_R(x_0)} |u|^q \leq e^{CR} \qquad \text{for all sufficiently large } \, R >> 1
	\]
	with $C<\frac{p(q-p+1)^{1/p'}}{(p-1)^{1/p'}} \frac{\lambda^{1/p}}{k}$, then $u\equiv 0$.
\end{theorem}

We conclude this introduction with a few comments on some technical points. First, in all the results stated above, except for Theorem \ref{thm-in5}, $M$ is not explicitely assumed to be non-compact. Indeed, if $M$ is compact (without boundary) and $u$ satisfies
\[
	Lu \geq f \geq 0 \qquad \text{on } \, \Omega_+
\]
for some measurable $f$, then necessarily $f=0$ and $A(x,u,\nabla u) = 0$ a.e.~on $\Omega_+$ (see Lemma \ref{cpt} in Section \ref{sec2}). Hence, in the assumptions of Theorems \ref{pL-in}, \ref{2L-in}, \ref{main_in} and \ref{thm-in4}, $M$ is necessarily non-compact. Secondly, in all our results we do not make additional regularity assumptions on the subsolutions beside their belonging to the appropriate Sobolev class $W^{1,p}_\loc(M)$. Since we do not know if Sobolev subsolutions of possibly degenerate equations of the form
\[
	\div A(x,u,\nabla u) = V|u|^{p-2}u
\]
are always locally essentially upper bounded (that is, if they necessarily satisfy $u_+\in L^\infty_\loc(M)$), in some of our arguments we have to follow more winding roads using approximation procedures.

\bigskip

The paper is organized as follows. In Section \ref{sec2} we collect the notation and all relevant definitions. In section \ref{sec3} we prove the main Theorem \ref{main_in} and we provide examples showing sharpness of the constants in the statements. Section \ref{sec4} is devoted to the proof of Theorem \ref{thm-in5}, from which Theorem \ref{thm-in4} can be easily deduced (see Corollary \ref{parq_1} and Remark \ref{parq_2}).

\bigskip

Comparison results and the case $p=1$ will appear in a forthcoming paper.

\section{Definitions and notation} \label{sec2}

Throughout this paper, $M$ will always be a connected Riemannian manifold withouth boundary. We denote by $TM$ its tangent bundle and by $\metric$ its Riemannian metric. For any $p\in(1,+\infty)$ we also denote by $W^{1,p}_\loc(M)$ the space of Sobolev functions $u$ whose restrictions to any relatively compact set $\Omega\subseteq M$ belong to $W^{1,p}(\Omega)$. This is equivalent to requiring that $u\circ\psi^{-1} \in W^{1,p}_\loc(\psi(U))$ for any local chart $\psi : U \subseteq M \to \R^m$, where $m=\dim M$. We also denote by $W^{1,p}_c(M)$ the subspace of $W^{1,p}_\loc(M)$ consisting of functions with compact support.

We consider quasilinear differential operators $L$ in divergence form weakly defined on functions $u\in W^{1,p}_\loc(M)$ by
\begin{equation} \label{Ldef}
	Lu (x) = \div A(x,u,\nabla u) \, .
\end{equation}
Here $A : \RR\times TM \to TM$ is a function such that
\[
	A(x,s,\xi) \in T_x M \qquad \forall \, x\in M, \, s\in\RR , \, \xi \in T_x M
\]
and whose local representation $\tilde A : \psi(U) \times \R \times \R^m \to \R^m$ in any chart $\psi : U \subseteq M \to \R^m$ satisfies the Carath\'eodory conditions
\begin{itemize}
	\item $\tilde A(y,\cdot,\cdot)$ is continuous for a.e.~$y\in\psi(U)$
	\item $\tilde A(\cdot,s,v)$ is measurable for every $(s,v)\in\R\times\R^m$.
\end{itemize}
(The representation $\tilde A$ is defined by
\[
	\tilde A(\psi(x),s,v) := A\left(x,s,\sum_{i=1}^m v^i \!\! \left.\frac{\partial}{\partial y^i}\right|_{x}\right) \qquad \forall \, x \in U, \, s \in \R, \, v = (v^1,\dots,v^m) \in \R^m
\]
where $y^1,\dots,y^m$ are the coordinates induced by $\psi$.) In particular, these conditions on $\tilde A$ are satisfied whenever $A$ is a continuous function of its arguments. Following terminology from \cite[Definition 2.1]{dam17}, we say that $A$ and the corresponding operator $L$ given by \eqref{Ldef} are \textit{weakly-$p$-coercive} for some $p\in(1,+\infty)$ if
\begin{align}
	\label{wpC1}
	\langle A(x,s,\xi),\xi \rangle \geq 0 & \qquad \forall \, x \in M, \, s \in \RR, \, \xi \in T_x M \\
	\label{wpC3}
	|A(x,s,\xi)| \leq k \langle A(x,s,\xi),\xi \rangle^{\frac{p-1}{p}} & \qquad \forall \, x \in M, \, s \in \RR, \, \xi \in T_x M \\
	\intertext{for some constant $k>0$ that we will call the \textit{coercivity constant} of $A$. Note that the above conditions imply that}
	\label{wpC4}
	|A(x,s,\xi)| \leq k^p |\xi|^{p-1} & \qquad \forall \, x \in M, \, s\in\R, \, \xi \in T_x M \, . \\
	\intertext{Indeed, this is clearly true when $A(x,s,\xi)=0$; otherwise, by Cauchy-Schwarz inequality and \eqref{wpC3} we have $|A(x,s,\xi)|^p \leq k^p |A(x,s,\xi)|^{p-1}|\xi|^{p-1}$, and then \eqref{wpC4} follows dividing both sides by $|A(x,s,\xi)|^{p-1}$. In particular, we have}
	\label{wpC2}
	A(x,s,0) = 0 & \qquad \forall \, x \in M, \, s \in \RR \, .
\end{align}
On the other hand, in general we do not assume non-degeneracy of $A$, that is, we do not assume that $A(x,s,\xi)\neq 0$ when $\xi\neq0$.

Let $A$ be a weakly-$p$-coercive function for some $p\in(1,+\infty)$. For any given $u\in W^{1,p}_\loc(M)$ and any $s_0\in\R$ we set
\[
	\Omega_{s_0} := \{ x \in M : u(x) > s_0 \}
\]
and for any non-negative measurable $f:M\to[0,+\infty]$ we say that $u$ satisfies
\begin{equation} \label{LufO}
	Lu \geq f \qquad \text{(weakly) on } \, \Omega_{s_0}
\end{equation}
if
\begin{equation} \label{Ltest}
	-\int_M \langle A(x,u,\nabla u),\nabla\varphi\rangle \geq \int_M f\varphi \qquad \forall \, \varphi \in D^+(\Omega_{s_0})
\end{equation}
where
\begin{align*}
	D^+(\Omega_{s_0}) := \{ \varphi\in W^{1,p}_c(M) : \; & \varphi\geq 0 \text{ on } M , \\
	& \varphi=0 \text{ and } \nabla\varphi=0 \text{ a.e.~on } M\setminus\Omega_{s_0} \} \, .
\end{align*}
We remark that our assumptions on $A$ and $u$ imply that $|A(x,u,\nabla u)|\in L^{p'}_\loc(M)$, with $p'=\frac{p}{p-1}$ the exponent conjugate to $p$, and that $\langle A(x,u,\nabla u),\nabla\varphi\rangle$ is measurable for each $\varphi\in D^+(\Omega_{s_0})$ (see for instance \cite[Lemma 2.4]{r18}). Hence, the LHS of \eqref{Ltest} is well defined and finite for each $\varphi\in D^+(\Omega_{s_0})$.

The next lemma justifies our focus on complete, non-compact manifolds in the introduction and in the following sections.

\begin{lemma} \label{cpt}
	Let $M$ be a compact manifold without boundary, $p\in(1,+\infty)$ and $L$ a weakly-$p$-coercive operator as in \eqref{Ldef}. If $u\in W^{1,p}(M)$ satisfies
	\[
		Lu \geq f \geq 0 \qquad \text{on } \, \Omega_{s_0} := \{u>s_0\}
	\]
	for some measurable $f : M \to \R$ and some $s_0\in\R$, then
	\begin{equation} \label{fA0}
		f = 0 \quad \text{and} \quad A(x,u,\nabla u) = 0 \qquad \text{a.e.~on } \, \Omega_{s_0} \, .
	\end{equation}
\end{lemma}

\begin{proof}
	Considering the test function $\varphi = (u-s_0)_+ \in D^+(\Omega_{s_0})$ we have
	\[
		\int_{\Omega_{s_0}} \langle A(x,u,\nabla u),\nabla u\rangle \leq -\int_{\Omega_{s_0}} (u-s_0)_+ f \leq 0
	\]
	and by the weak coercivity condition \eqref{wpC3} we obtain
	\[
		\int_{\Omega_{s_0}} |A(x,u,\nabla u)|^{\frac{p}{p-1}} \leq 0 \, .
	\]
	By non-negativity of $f$ and of $|\,\cdot\,|$, this immediately yields.
\end{proof}

Lastly, we precise the following terminology. For an open interval $I\subseteq\R$ we say that a function $F : I \to \R$ is piecewise $C^1$ if $F$ is continuous on $I$ and there exists a discrete (possibly empty) set $E\subseteq I$ such that
\begin{align*}
	i) & \quad F' \, \text{ exists and is continuous on $I\setminus E$} \\
	ii) & \quad \forall \, a\in E \quad \lim_{x\to a^-} F'(x) \quad \text{and} \quad \lim_{x\to a^-} F'(x) \quad \text{exist and are finite.}
\end{align*}
If $u\in W^{1,p}_\loc(M)$ with $u(M)\subseteq I$ and $F'$ is bounded on $I\setminus E$, then by Stampacchia's lemma the function $v = F(u)$ is also in $W^{1,p}_\loc(M)$ and
\[
	\nabla v = \begin{cases}
		F'(u)\nabla u & \quad \text{a.e.~on } \, M\setminus u^{-1}(E) \\
		0 & \quad \text{a.e.~on } \, u^{-1}(E) \, ,
	\end{cases}
\]
see for instance Theorem 7.8 in \cite{gt}. (Here and in the following statements, ``a.e.''~always referes to the $m$-dimensional Riemannian volume measure of $M$.) Since $\nabla u = 0$ a.e.~on each level set of $u$, we can further write
\[
	\nabla v = F'(u)\nabla u \qquad \text{a.e.~on } \, M \, .
\]

\section{Proof of the main theorem} \label{sec3}

The aim of this section is to prove the main Theorem \ref{wp-main} below, which is slightly more general than Theorem \ref{main_in} from the Introduction. To do so, we have to collect some preliminary lemmas about functions $u$ satisfying $Lu\geq 0$ on some superlevel set $\Omega_{s_0} := \{x\in M : u(x)>s_0\}$, $s_0\in\R$. Note that for the validity of the following lemmas it is not necessary to assume that $|\Omega_{s_0}|>0$, that is, $s_0$ may be a priori larger than or equal to $\esssup_M u$ (in which case it is clearly true that $Lu\geq 0$ on $\Omega_{s_0}$ in the sense of \eqref{Ltest}, and the thesis of each lemma holds trivially).

\begin{lemma} \label{lem_Fm}
	Let $M$ be a Riemannian manifold, $p>1$ and $L$ a weakly-$p$-coercive operator as in \eqref{Ldef} with coercivity constant $k>0$. Let $u\in W^{1,p}_\loc(M)$ satisfy
	\begin{equation} \label{Luf}
		Lu \geq f \geq 0 \qquad \text{on } \, \Omega_{s_0} := \{ x \in M : u(x) > s_0 \}
	\end{equation}
	for some $s_0\in\R$ and some measurable $f : M \to \R$. Let $F$ be a non-negative, non-decreasing, piecewise $C^1$ function on $(0,+\infty)$. Then for every $0\leq\eta\in C^\infty_c(M)$
	\begin{equation} \label{key_in}
		\int_{\Omega_{s_0}} F(w)|A_u||\nabla\eta| \geq k^{-p'} \int_{\Omega_{s_0}} \eta F'(w)|A_u|^{p'} + \int_{\Omega_{s_0}} \eta F(w) f \, ,
	\end{equation}
	where $w:=(u-s_0)_+$, $A_u:=A(x,u,\nabla u)$ and $p'=\frac{p}{p-1}$.
\end{lemma}

\begin{proof}
	Let $0\leq\eta\in C^\infty_c(M)$ be given and let
	\[
		w := (u - s_1)_+ \in W^{1,p}_\loc(M) \, , \qquad A_u := A(x,u,\nabla u)
	\]
	as in the statement. Let $\lambda\in C^\infty(\R)$ be such that
	\begin{equation} \label{lam}
		\lambda(s) = 0 \quad \text{if } \, s \leq 1 \, , \qquad \lambda(s) = 1 \quad \text{if } \, s \geq 2 \, , \qquad \lambda'\geq 0 \quad \text{on } \, \R
	\end{equation}
	and for any $\eps>0$ define $\lambda_\eps \in C^\infty(\R)$ by
	\begin{equation} \label{le_def}
		\lambda_\eps(s) := \lambda(s/\eps) \, .
	\end{equation}
	Clearly we have
	\begin{equation} \label{leps}
		0\leq\lambda_\eps\leq\mathbf{1}_{(0,+\infty)} \quad \forall \, \eps>0 \qquad \text{and} \qquad \lambda_\eps \nearrow \mathbf{1}_{(0,+\infty)} \quad \text{as } \, \eps \to 0^+ \, ,
	\end{equation}
	where $\mathbf{1}$ denotes the indicator function and $\nearrow$ denotes monotone convergence from below. Let $h>0$ be fixed and for any $\eps\in(0,h/2)$ let $F_{\eps,h} : \R \to [0,+\infty)$ be given by
	\[
		F_{\eps,h}(s) = \begin{cases}
			0 & \quad \text{if } \, s < 0 \\
			\lambda_\eps(s) F(s) & \quad \text{if } \, 0\leq s < h \\
			F(h) & \quad \text{if } \, s \geq h \, .
		\end{cases}
	\]
	By our choice of $\lambda$ and our assumptions on $F$, the function $F_{\eps,h}$ is non-negative, non-decreasing, piecewise $C^1$ on $\R$ (with an additional corner point at $s=h$) and globally Lipschitz, so $F_{\eps,h}(w)\in W^{1,p}_\loc(M)$ with
	\[
		\nabla F_{\eps,h}(w) = F_{\eps,h}'(w) \nabla u \qquad \text{a.e.~on } \, M \, .
	\]
	In particular we have
	\[
		F_{\eps,h}'(s) = \begin{cases}
			\lambda_\eps'(s) F(s) + \lambda_\eps(s) F'(s) \geq \lambda_\eps(s) F'(s) & \qquad \text{if } \eps < s < h \\
			0 & \qquad \text{if } \, s \leq \eps \, \text{ or } \, s > h \, .
		\end{cases}
	\]
	Set
	\[
		\varphi = \varphi_{\eps,h} := \eta F_{\eps,h}(w) \, .
	\]
	We have $0\leq \varphi \in W^{1,p}_c(M)$ and by the choice of $\lambda_\eps$ we also have that $\varphi$ vanish outside $\{w>0\} \equiv \Omega_{s_0}$. So $\varphi$ is an admissible test function for \eqref{Ltest} and we have
	\begin{equation} \label{in0F}
		- \int_M \langle A_u,\nabla\varphi \rangle \geq \int_M f \varphi \, .
	\end{equation}
	By direct computation and using that $\eta F(w)\lambda_\eps'(w)\langle A_u,\nabla u\rangle \geq 0$ by our assumptions on $\lambda_\eps$, $F$, $\eta$ and $A$, together with weak $p$-coercivity \eqref{wpC3} of $A$ and Cauchy-Schwarz inequality we have
	\begin{align*}
		\langle A_u,\nabla\varphi \rangle & = F_{\eps,h}(w) \langle A_u,\nabla\eta \rangle + \eta F_{\eps,h}'(w) \langle A_u,\nabla u \rangle \\
		& \geq F_{\eps,h}(w) \langle A_u,\nabla\eta \rangle + \eta F'(w) \langle A_u,\nabla u \rangle \lambda_\eps(w) \mathbf{1}_{\{\eps<w<h\}} \\
		& \geq - F_{\eps,h}(w) |A_u| |\nabla\eta| + k^{-p'} \eta F'(w) |A_u|^{p'} \lambda_\eps(w) \mathbf{1}_{\{\eps<w<h\}} \, .
	\end{align*}
	We substitute into \eqref{in0F} and rearrange terms to get
	\[
		\int_{\Omega_{s_0}} F_{\eps,h}(w)|A_u||\nabla\eta| \geq k^{-p'} \int_{\{\eps<w<h\}} \eta\lambda_{\eps}(w) F'(w)|A_u|^{p'} + \int_{\Omega_{s_0}} \eta F_{\eps,h}(w) f \, .
	\]
	Using non-negativity of $F$, $F'$, $f$, $\eta$, monotonicity of $F$ and \eqref{leps}, by the monotone convergence theorem we get
	\begin{align*}
		\lim_{\substack{\eps\to0^+\\h\to+\infty}} \int_{\Omega_{s_0}} F_{\eps,h}(w)|A_u||\nabla\eta| & = \int_{\Omega_{s_0}} F(w)|A_u||\nabla\eta| \\
		\lim_{\substack{\eps\to0^+\\h\to+\infty}} \int_{\{\eps<w<h\}} \eta \lambda_{\eps}(w) F'(w) |A_u|^{p'} & = \int_{\Omega_{s_0}} \eta F'(w) |A_u|^{p'} \\
		\lim_{\substack{\eps\to0^+\\h\to+\infty}} \int_{\Omega_{s_0}} \eta F_{\eps,h}(w) f & = \int_{\Omega_{s_0}} \eta F(w) f
	\end{align*}
	and then we obtain \eqref{s_alt}.
\end{proof}

We underline that the LHS of \eqref{key_in} can be further estimated from above via Young's inequality in two different ways, both useful in what will follow.

\textbf{(1)} Suppose that $F'>0$ on $(0,+\infty)$. By H\"older's and Young's inequalities with conjugate exponents $p$ and $p'$, for any $\sigma>0$ we get
\begin{equation} \label{FFgYo}
	\begin{split}
		\int_{\Omega_{s_0}} F(w)|A_u||\nabla\eta| & \leq \left( \int_{\Omega_{s_0}} \frac{[F(w)]^p}{[F'(w)]^{p-1}}|\nabla\eta| \right)^{1/p} \left( \int_{\Omega_{s_0}} F'(w)|A_u|^{p'}|\nabla\eta| \right)^{1/p'} \\
		& \leq \frac{\sigma^p}{p}\int_{\Omega_{s_0}} \frac{[F(w)]^p}{[F'(w)]^{p-1}}|\nabla\eta| + \frac{\sigma^{-p'}}{p'}\int_{\Omega_{s_0}} F'(w)|A_u|^{p'}|\nabla\eta| \, .
	\end{split}
\end{equation}

\textbf{(2)} If $0\leq\psi\in C^\infty_c(M)$, then applying \eqref{key_in} with $\eta:=\psi^p\in C^\infty_c(M)$ we get
\begin{equation}
	p \int_{\Omega_{s_0}} \psi^{p-1} F(w)|A_u||\nabla\psi| \geq k^{-p'} \int_{\Omega_{s_0}} \psi^p F'(w)|A_u|^{p'} + \int_{\Omega_{s_0}} \psi^p F(w)f
\end{equation}
and by Young's inequality we have, again for any $\sigma>0$,
\begin{equation} \label{FFgYz}
	p \int_{\Omega_{s_0}} \psi^{p-1} F(w)|A_u||\nabla\psi| \leq \frac{p^p \sigma^p}{p} \int_{\Omega_{s_0}} \frac{[F(w)]^p}{[F'(w)]^{p-1}} |\nabla\psi|^p + \frac{\sigma^{-p'}}{p'} \int_{\Omega_{s_0}} \psi^p F'(w)|A_u|^{p'} \, .
\end{equation}

By suitably choosing $\sigma$ in \eqref{FFgYz} and rearranging terms we deduce the following

\begin{lemma} \label{lem_0}
	In the assumptions of Lemma \ref{lem_Fm}, if
	\begin{equation} \label{112.1}
		F'(w)|A_u|^{p'}\mathbf{1}_{\Omega_{s_0}} \in L^1_\loc(M)
	\end{equation}
	then for any $\eps>0$ and for any $0\leq\eta\in C^\infty_c(M)$ we have
	\begin{equation} \label{112}
		\frac{k^p(p-1)^{p-1}}{\eps^{p-1}} \int_{\Omega_{s_0}} \frac{[F(w)]^p}{[F'(w)]^{p-1}} |\nabla\eta|^p \geq (1-\eps)k^{-p'} \int_{\Omega_{s_0}} \eta^p F'(w)|A_u|^{p'} + \int_{\Omega_{s_0}} \eta^p F(w) f \, .
	\end{equation}
	In particular, \eqref{112.1} holds under one of the following assumptions:
	\begin{itemize}
		\item[(a)] $F(s) = O(s)$ as $s\to+\infty$
		\item[(b)] $u_+\in L^r_\loc(M)$ and $F(s)=O(s^{r/p})$ as $s\to+\infty$, for some $r>p$
		\item[(c)] $u_+\in L^\infty_\loc(M)$.
	\end{itemize}
\end{lemma}

\begin{proof}
	If $\eps>0$ is given then for $\sigma = (\eps p')^{-1/p'} k$ we have
	\[
		\frac{\sigma^{-p'}}{p'} = \eps k^{-p'} \, , \qquad \frac{p^p\sigma^p}{p} = \frac{k^p(p-1)^{p-1}}{\eps^{p-1}}
	\]
	and then from \eqref{FFgYz} we get
	\begin{equation} \label{in_rear}
		\begin{split}
			\frac{k^p(p-1)^{p-1}}{\eps^{p-1}} \int_{\Omega_{s_0}} \frac{[F(w)]^p}{[F'(w)]^{p-1}} & |\nabla\eta|^p + \eps k^{-p'} \int_{\Omega_{s_0}} \eta^p F'(w)|A_u|^{p'} \\
			& \geq k^{-p'} \int_{\Omega_{s_0}} \eta^p F'(w)|A_u|^{p'} + \int_{\Omega_{s_0}} \eta^p F(w)f \, .
		\end{split}
	\end{equation}
	In the assumption \eqref{112.1} we can rearrange terms to obtain \eqref{112}. In view of \eqref{key_in} and since $f\geq0$ on $\Omega_{s_0}$, condition \eqref{112.1} is automatically satisfied if $F(w)|A_u\mathbf{1}_{\Omega_{s_0}}|\in L^1_\loc(M)$. In particular this is always the case if $F(w)\mathbf{1}_{\Omega_{s_0}}\in L^p_\loc(M)$, because then $F(w)|A_u|\mathbf{1}_{\Omega_{s_0}}\in L^1_\loc(M)$ by H\"older inequality (recall that $u\in W^{1,p}_\loc(M)$, so $|A_u|\leq k^p|\nabla u|^{p-1} \in L^{p'}_\loc(M)$), and condition $F(w)\mathbf{1}_{\Omega_{s_0}}\in L^p_\loc(M)$ is in turn satisfied in either one of the cases (a), (b) or (c). 
\end{proof}

A case that will be relevant for our subsequent discussion is where $u_+\in L^q_\loc(M)$ and $F(s)=s^{q-p+1}$ for some $q\in(p-1,+\infty)$. In this setting the desired inequality takes the form
\[
	\frac{k^p(p-1)^{p-1}}{\eps^{p-1}\gamma^{p-1}}\int_{\Omega_{s_0}} w^q|\nabla\eta|^p \geq (1-\eps)k^{-p'} \int_{\Omega_{s_0}} \eta^p w^{q-p}|A_u|^{p'} + \int_{\Omega_{s_0}} \eta^p w^p f
\]
where $\gamma:=q-p+1\in(0,+\infty)$. Note that for $p-1<q\leq p$ we have $0<\gamma\leq 1$, hence $F(s)=s^{q-p+1}=s^\gamma=O(s)$ and this scenario is covered by alternative (a) in Lemma \ref{lem_0}, while for $q>p$ (and without assuming $u_+\in L^\infty_\loc(M)$) we cannot refer to (b) or (c).

\begin{lemma} \label{lem_Fgamma}
	Let $M$ be a Riemannian manifold, $p\in(1,+\infty)$ and $L$ a weakly-$p$-coercive operator as in \eqref{Ldef} with coercivity constant $k>0$. Let $u\in W^{1,p}_\loc(M)$ satisfy
	\begin{equation} \label{Luf}
		Lu \geq f \geq 0 \qquad \text{on } \, \Omega_{s_0} := \{ x \in M : u(x) > s_0 \}
	\end{equation}
	for some $s_0\in\R$ and some measurable $f : M \to \R$. Let $w:=(u-s_0)_+$ and $A_u:=A(x,u,\nabla u)$. Then for any $q\in(p-1,+\infty)$ and for every $0\leq\eta\in C^\infty_c(M)$
	\begin{equation} \label{Fg0}
		\frac{k^p(p-1)^{p-1}}{\eps^{p-1}\min\{1,\gamma^{p-1}\}}\int_{\Omega_{s_0}} w^q|\nabla\eta|^p \geq (1-\eps)\gamma k^{-p'} \int_{\Omega_{s_0}} \eta^p w^{q-p}|A_u|^{p'} + \int_{\Omega_{s_0}} \eta^p w^{q-p+1} f
	\end{equation}
	where $\gamma:=q-p+1$. If $u_+\in L^q_\loc(M)$, this can be strengthened to
	\begin{equation} \label{Fg1}
		\frac{k^p(p-1)^{p-1}}{\eps^{p-1}\gamma^{p-1}}\int_{\Omega_{s_0}} w^q|\nabla\eta|^p \geq (1-\eps)\gamma k^{-p'} \int_{\Omega_{s_0}} \eta^p w^{q-p}|A_u|^{p'} + \int_{\Omega_{s_0}} \eta^p w^{q-p+1} f \, .
	\end{equation}
	In particular, if $u_+\in L^\infty_\loc(M)$ then this holds for any $q\in(p-1,+\infty)$.
\end{lemma}

\begin{proof}
	Let $0\leq\eta\in C^\infty_c(M)$, $q\in(p-1,+\infty)$ be given and set $F(s)=s^\gamma$ for $s>0$, where $\gamma:=q-p+1$ as in the statement of the Lemma.
	
	If $p-1<q\leq p$ then $0<\gamma\leq 1$ and by Lemma \ref{lem_0} we have the validity of \eqref{Fg1} for any $\eps\in(0,1]$. (Note that in this case \eqref{Fg0} and \eqref{Fg1} coincide.) 
	
	If $q>p$ then we proceed by approximating $F$ from below with globally Lipschitz functions. For any $h>0$ let $F_h : (0,+\infty) \to (0,+\infty)$ be defined by
	\[
		F_h(s) = \begin{cases}
			s^\gamma & \text{if } \, 0 < s \leq h \\
			h^{\gamma-1} s & \text{if } \, s > h \, .
		\end{cases}
	\]
	Then $F_h$ is piecewise smooth with a corner point at $s=h$ and satisfies $F_h(s) = O(s)$ as $s\to+\infty$, so by Lemma \ref{lem_0} we have
	\[
		\frac{k^p(p-1)^{p-1}}{\eps^{p-1}} \int_{\Omega_{s_0}} \frac{[F_h(w)]^p}{[F_h'(w)]^{p-1}}|\nabla\eta|^p \geq (1-\eps) k^{-p'} \int_{\Omega_{s_0}} \eta^p F_h'(w)|A_u|^{p'} + \int_{\Omega_{s_0}} \eta^p F_h(w)f \, .
	\]
	By direct computation we have
	\begin{alignat*}{2}
		F_h'(w)|A_u|^{p'} & = \gamma w^{q-p} |A_u|^{p'} \mathbf{1}_{\{0<w\leq h\}} + h^{q-p} |A_u|^{p'} \mathbf{1}_{\{w>h\}} && \qquad \text{a.e.~on } \, \Omega_{s_0} \\
		\frac{[F_h(w)]^p}{[F_h'(w)]^{p-1}} & \leq \frac{w^q}{\gamma^{p-1}} \mathbf{1}_{\{0<w\leq h\}} + h^{q-p} w^p \mathbf{1}_{\{w>h\}} \leq w^q && \qquad \text{on } \, \Omega_{s_0}
	\end{alignat*}
	We substitute the second estimate into the previous inequality to obtain
	\[
		\frac{k^p(p-1)^{p-1}}{\eps^{p-1}} \int_{\Omega_{s_0}} w^q|\nabla\eta|^p \geq (1-\eps) k^{-p'} \int_{\Omega_{s_0}} \eta^p F_h'(w)|A_u|^{p'} + \int_{\Omega_{s_0}} \eta^p F_h(w)f
	\]
	and then letting $h\to+\infty$ we get, by the monotone convergence theorem,
	\[
		\frac{k^p(p-1)^{p-1}}{\eps^{p-1}} \int_{\Omega_{s_0}} w^q|\nabla\eta|^p \geq (1-\eps) \gamma k^{-p'} \int_{\Omega_{s_0}} \eta^p w^{q-p}|A_u|^{p'} + \int_{\Omega_{s_0}} \eta^p w^{q-p+1} f
	\]
	proving \eqref{Fg0}.
	
	If additionally $u_+\in L^q_\loc(M)$, then for any given $0\leq\eta\in C^\infty_c(M)$
	\[
		\int_{\Omega_{s_0}} \frac{[F(w)]^p}{[F'(w)]^{p-1}}|\nabla\eta|^p \equiv \frac{1}{\gamma^{p-1}} \int_{\Omega_{s_0}} w^q |\nabla\eta|^p < +\infty
	\]
	and from \eqref{Fg0} applied for any $\eps\in(0,1)$ we deduce (since $f\geq 0$) that also
	\[
		\int_{\Omega_{s_0}} \eta^p F'(w)|A_u|^{p'} \equiv \gamma \int_{\Omega_{s_0}} \eta^p w^{q-p}|A_u|^{p'} < +\infty \, .
	\]
	Since this holds for any $0\leq\eta\in C^\infty_c(M)$ we have that $F'(w)|A_u|^{p'}\mathbf{1}_{\Omega_{s_0}}\in L^1_\loc(M)$, that is, the hypothesis \eqref{112.1} in Lemma \ref{lem_0} is satisfied, and then \eqref{Fg1} directly follows from that lemma.
\end{proof}

We briefly comment on the condition $u_+\in L^\infty_\loc(M)$. If the function $A$ satisfies the additional coercivity condition
\begin{equation} \label{wpC5}
	|A(x,s,\xi)| \geq k_2 |\xi|^{p-1} \qquad \forall \, x \in M, \, s\in \R, \, \xi \in T_x M
\end{equation}
for some constant $k_2>0$ (note that this is the case for the $p$-Laplacian $L=\Delta_p$) then subsolutions of $Lu = 0$ on $M$ are locally essentially bounded above, that is, condition $u_+\in L^\infty_\loc(M)$ is automatically satisfied for any $u\in W^{1,p}_\loc(M)$ satisfying
\begin{equation} \label{Lu0w}
	Lu\geq 0 \qquad \text{weakly on } \, M \, .
\end{equation}
More generally, $u_+\in L^\infty_\loc(M)$ holds for functions $u\in W^{1,p}_\loc(M)$ such that, for some $s_0\in\R$, the truncation $w := (u-s_0)_+$ satisfies $Lw\geq 0$ weakly on $M$.

\begin{proposition}
	Let $M$ be a Riemannian manifold, $p>1$ and $L$ as in \eqref{Ldef} a weakly-$p$-coercive operator for which \eqref{wpC5} holds. Let $u\in W^{1,p}_\loc(M)$ satisfy
	\begin{equation} \label{Lu0}
		L(u-s_0)_+ \geq 0 \qquad \text{weakly on } \, M
	\end{equation}
	for some $s_0\in\R$. Then $u_+ \in L^\infty_\loc(M)$.
\end{proposition}

\begin{proof}[Sketch of proof]
	For $p>\dim M$ the thesis holds because $W^{1,p}_\loc(M) \subseteq C(M)$ by (local) Sobolev embeddings, while for $1<p\leq\dim M$ the statement can be proved by Moser iteration technique, using the Caccioppoli-type inequality
	\[
		\frac{2^p(p-1)^{p-1}k^{pp'}}{\gamma \min\{1,\gamma^{p-1}\}} \int_M |\nabla\eta|^p (u-s_0)_+^q \geq k_2^{p'} \int_M \eta^p (u-s_0)_+^{q-p} |\nabla u|^p
	\]
	obtained by \eqref{Fg0} (with the choices $\eps=1/2$ and $f=0$) and \eqref{wpC5}, together with the fact that every point $x\in M$ has a relatively compact neighbourhood $U\subseteq M$ on which a Sobolev inequality holds. In fact, the Moser technique can be used to prove that $(u-s_0)_+ \in L^\infty_\loc(M)$, from which $u_+\in L^\infty_\loc(M)$ immediately follows.
\end{proof}

Since the argument above is of local nature, clearly it also applies in case \eqref{wpC5} is satisfied with $k_2 : M \to (0,+\infty)$ a continuous function possibly decaying to zero at infinity. However, in our analysis we are not assuming coercivity conditions of the form \eqref{wpC5}, and in fact we don't know whether a function $u\in W^{1,p}_\loc(M)$ such that $Lu\geq 0$ on some superlevel set $\{u>s_0\}$, with $L$ only satisfying assumptions \eqref{wpC1}-\eqref{wpC3} from Section \ref{sec2}, is necessarily locally upper bounded.

\bigskip

We are now ready to state and prove the main theorem of this section.

\begin{theorem} \label{wp-main}
	Let $M$ be a complete Riemannian manifold, $p\in(1,+\infty)$ and $L$ a weakly-$p$-coercive operator as in \eqref{Ldef} with coercivity constant $k>0$. Let $\lambda>0$, $\mu\in[0,p]$ and $V : M \to (0,+\infty)$ be a continuous function satisfying
	\begin{equation} \label{Vlm1}
		\begin{array} {r@{\;}c@{\;}ll}
			V & \geq & \lambda & \qquad \text{if } \, \mu = 0 \\[0.2cm]
			\displaystyle\liminf_{x\to\infty} \, [\dist(x,o)^\mu V(x)] & \geq & \lambda \quad \text{for some } \, o\in M & \qquad \text{if } \, \mu \in (0,p] \, .
		\end{array}
	\end{equation}
	Let $u\in W^{1,p}_\loc(M)$ satisfy, for some $0\leq s_0<\esssup_M u$,
	\[
		Lu \geq V u^{p-1} \qquad \text{on } \, \Omega_{s_0} := \{ x\in M : u(x) > s_0 \} \, .
	\]
	Then for any $x_0\in M$ and $q\in(p-1,+\infty)$ we have
	\begin{alignat}{2}
		\label{wp-0}
		\liminf_{R\to+\infty} \, \frac{1-\frac{\mu}{p}}{R^{1-\frac{\mu}{p}}}
		\log\displaystyle\int_{B_R} (u-s_0)_+^q & \geq C_0 > 0 && \qquad \text{if } \, \mu \in [0,p) \\
		\label{wp-1}
		\liminf_{R\to+\infty} \, \frac{1}{\log R}
		\log\displaystyle\int_{B_R} (u-s_0)_+^q & \geq C_1 > p && \qquad \text{if } \, \mu = p
	\end{alignat}
	where $C_0$ and $C_1$ are determined by
	\[
		C_0 := \frac{p(q-p+1)^{1/p'}\lambda^{1/p}}{(p-1)^{1/p'}k} \, , \qquad C_1^{1/p}(C_1-p)^{1/p'} = C_0 \, .
	\]
	Moreover, in case $\mu=p$ we have
	\begin{equation} \label{wp-1b}
		\lim_{R\to+\infty} \frac{1}{\log R} \log\int_{B_R}(u-s_0)_+^q \geq C_0 + p
	\end{equation}
	whenever the limit on the LHS exists.
\end{theorem}

\begin{remark}
	Note that $C_0+p>C_1>C_0$ always.
\end{remark}

\begin{proof}
	Let us set $w:=(u-s_0)_+$ and $A_u:=A(x,u,\nabla u)$. Let $x_0\in M$ and $q\in(p-1,+\infty)$ be given. For the sake of brevity, for any $R>0$ we shall write $B_R$ to denote the geodesic ball $B_R(x_0)$. Without loss of generality we can assume $w^q \in L^1_\loc(M)$, since otherwise $\int_{B_R} w^q = +\infty$ for each sufficiently large $R>0$ and the conclusion is trivial. Note that under this assumption we also have $w^{q-p}|A_u|^{p'}\mathbf{1}_{\Omega_{s_0}}\in L^1_\loc(M)$, as a consequence of \eqref{Fg1} in Lemma \ref{lem_Fgamma}. Let $G,H:(0,+\infty) \to [0,+\infty)$ be defined by
	\begin{equation} \label{GHd}
		G(t) := \int_{B_t} w^q \, , \qquad H(t) := \int_{\Omega_{s_0}\cap B_t} w^{q-p}|A_u|^{p'} \, .
	\end{equation}
	By the previous observation, the functions $G$ and $H$ are well defined, non-decreasing and absolutely continuous on any compact interval contained in $(0,+\infty)$. In particular, they are differentiable a.e.~on $(0,+\infty)$.
	
	Since $s_0\geq 0$, we have $u^{p-1}\geq w^{p-1}$ on $\Omega_{s_0}$. Then by applying Lemma \ref{lem_Fm} with the choices $F(s) = s^{q-p+1}$ and $f = V w^{p-1}$ we have
	\begin{equation} \label{in1}
		\int_M w^{q-p+1}|A_u||\nabla\eta| \geq \gamma k^{-p'} \int_{\Omega_{s_0}} \eta w^{q-p}|A_u|^{p'} + \int_M V \eta w^q
	\end{equation}
	for any $0\leq\eta\in C^\infty_c(M)$, where $\gamma:=q-p+1>0$, and applying Young's inequality as in \eqref{FFgYo} we have, for any $\sigma>0$,
	\begin{equation} \label{in2}
		\int_M w^{q-p+1}|A_u||\nabla\eta| \leq \frac{\sigma^p}{p} \int_M w^q |\nabla\eta| + \frac{\sigma^{-p'}}{p'} \int_{\Omega_{s_0}} w^{q-p} |A_u|^{p'}|\nabla\eta| \, .
	\end{equation}
	
	Let $\eps\in(0,\lambda)$ be given. By condition \eqref{Vlm1} and continuity and (strict) positivity of $V$, there exists $R_0=R_0(x_0,\eps)>0$ large enough so that
	\begin{equation} \label{b1R0b}
		V(x) \geq \frac{\lambda-\eps}{\dist(x,x_0)^\mu} \qquad \text{for all } \, x\in M \setminus B_{R_0}
	\end{equation}
	and
	\begin{equation} \label{b12R0}
		\inf_{B_R} V \geq \frac{\lambda-\eps}{R^\mu} \qquad \forall \, R > R_0 \, .
	\end{equation}
	(For $\mu=0$ this is clearly true since $V\geq\lambda$ everywhere on $M$ by assumption \eqref{Vlm1}. In case $\mu>0$, note that it is possible to first find $r_0>0$ such that
	\begin{equation} \label{b12r2}
		V(x) \geq \frac{\lambda-\eps}{\dist(x,x_0)^\mu} \qquad \text{for all } \, x \in M\setminus B_{r_0}
	\end{equation}
	since from \eqref{Vlm1} and the triangle inequality we have
	\[
		\liminf_{x\to\infty} \, [\dist(x,x_0)^\mu V(x)] \geq \lambda \, ,
	\]
	and then for any $R>r_0$ we get
	\begin{equation} \label{b12r3}
		\inf_{B_R} V \geq \min\left\{ \inf_{B_{r_0}} V \, , \frac{\lambda-\eps}{R^\mu} \right\}.
	\end{equation}
	From the assumption that $V$ is continuous and strictly positive on $M$ we have $\inf_{B_{r_0}} V > 0$, so we can find $R_0\geq r_0$ such that $\inf_{B_{r_0}} V \geq (\lambda-\eps)/R_0^\mu$. Then for any $R>R_0$ the RHS in \eqref{b12r3} is just $(\lambda-\eps)/R^\mu$, and so \eqref{b1R0b}-\eqref{b12R0} hold for such $R_0$.)
	
	Let $t>R_0$ be a value for which $G'(t)$ and $H'(t)$ both exist. For any $0<\delta<t$ choose $\eta_\delta\in C^\infty_c(M)$ satisfying
	\[
	\begin{array}{rll}
		i) & \quad \eta_\delta\equiv 1 & \quad \text{on } \, B_{t-\delta} \, , \\[0.2cm]
		ii) & \quad \eta_\delta\equiv 0 & \quad \text{on } \, M\setminus B_t \, , \\[0.2cm]
		iii) & \quad 0 \leq \eta_\delta \leq 1 & \quad \text{on } \, B_t \setminus B_{t-\delta} \\[0.2cm]
		iv) & \quad |\nabla\eta_\delta| \leq \dfrac{1}{\delta} + 1 & \quad \text{on } \, M \, .
	\end{array}
	\]
	Since $|\nabla\eta_\delta|\leq (1+\delta^{-1}) \mathbf{1}_{B_R\setminus B_{R-\delta}}$ we have
	\[
		\int_M w^q|\nabla\eta_\delta| \leq \left(\frac{1}{\delta}+1\right) \int_{B_t\setminus B_{t-\delta}} w^q = (1+\delta) \frac{G(t)-G(t-\delta)}{\delta}
	\]
	and letting $\delta\searrow0$ we get
	\[
		\limsup_{\delta\to0^+} \int_M w^q |\nabla\eta_\delta| \leq G'(t) \, .
	\]
	Similarly, we have
	\[
		\limsup_{\delta\to 0^+} \int_{\Omega_{s_0}} w^{q-p}|A_u|^{p'}|\nabla\eta_\delta| \leq H'(t) \, .
	\]
	On the other hand, since $\eta_\delta=0$ on $M\setminus B_t$ and $\eta_\delta\to\mathbf{1}_{B_t}$ pointwise as $\delta\to 0$, by the dominated convergence theorem and also using \eqref{b12R0} we get
	\begin{align*}
		\lim_{\delta\to0^+} \int_{\Omega_{s_0}} \eta_\delta w^{q-p}|A_u|^{p'} & = \int_{\Omega_{s_0}\cap B_t} w^{q-p}|A_u|^{p'} = H(t) \\
		\lim_{\delta\to0^+} \int_M V \eta_\delta w^q & = \int_{B_t} V w^q \, .
	\end{align*}
	Thus, in view of \eqref{in1}-\eqref{in2} we have, for any $\sigma>0$,
	\begin{equation} \label{GH'0}
		\frac{\sigma^p}{p} G'(t) + \frac{\sigma^{-p'}}{p'} H'(t) \geq \int_{B_t} Vw^q + \gamma k^{-p'} H(t)
	\end{equation}
	and using \eqref{b12R0} to further estimate
	\[
		\int_{B_t} Vw^q \geq \frac{\lambda-\eps}{t^\mu} \int_{B_t} w^q = \frac{\lambda-\eps}{t^\mu} G(t)
	\]
	we obtain
	\[
		\frac{\sigma^p}{p} G'(t) + \frac{\sigma^{-p'}}{p'} H'(t) \geq \frac{\lambda-\eps}{t^\mu} G(t) + \gamma k^{-p'} H(t) \, .
	\]
	We apply the above reasoning to each value $t>R_0$ for which $G$ and $H$ are simultaneously differentiable to deduce that for any $\sigma : (0,+\infty) \to (0,+\infty)$
	\[
		\frac{[\sigma(t)]^p}{p} G'(t) + \frac{[\sigma(t)]^{-p'}}{p'} H'(t) \geq \frac{\lambda-\eps}{t^\mu} G(t) + \gamma k^{-p'} H(t) \qquad \text{for a.e.~} t > R_0
	\]
	that is, multiplying everything by $p[\sigma(t)]^{-p}$ and recalling that $p+p'=pp'$,
	\begin{equation} \label{GH'}
		G'(t) + \frac{p-1}{[\sigma(t)]^{pp'}} H'(t) \geq \frac{p(\lambda-\eps)}{[\sigma(t)]^p t^\mu} \left( G(t) + \frac{\gamma}{(\lambda-\eps)k^{p'}} t^\mu H(t) \right)
	\end{equation}
	for a.e.~$t>R_0$. We now consider separately the cases $\mu\in[0,p)$ and $\mu=p$.
	
	\medskip
	
	\noindent \textbf{Case $\mu\in[0,p)$.} Assume that $\mu\in[0,p)$. Choosing
	\begin{align*}
		c_1 = c_{1,\eps} & = (p-1)^{\frac{1}{pp'}} (\lambda-\eps)^{\frac{1}{pp'}} \gamma^{-\frac{1}{pp'}} k^{1/p} \\
		c_2 = c_{2,\eps} & = \frac{(p-1)}{c_1^{pp'}} \equiv \gamma (\lambda-\eps)^{-1} k^{-p'} \\
		c_3 = c_{3,\eps} & = \frac{p(\lambda-\eps)}{c_1^p} \equiv \frac{p\gamma^{1/p'}(\lambda-\eps)^{1/p}}{(p-1)^{1/p'}k} \\
		\sigma(t) & = c_1 t^{-\frac{\mu}{pp'}}
	\end{align*}
	we get
	\[
		G'(t) + c_2 t^\mu H'(t) \geq c_3 t^{-\frac{\mu}{p}} \left(G(t) + c_2 t^\mu H(t)\right) \qquad \text{for a.e.~} t>R_0 \, .
	\]
	Let $\Phi : (0,+\infty) \to [0,+\infty)$ be defined by
	\[
		\Phi(t) = G(t) + c_2 t^\mu H(t) \, .
	\]
	The function $\Phi$ is absolutely continuous on each compact subset of $(0,+\infty)$ with
	\begin{equation} \label{P'}
		\Phi'(t) = G'(t) + c_2 t^\mu H'(t) + \mu c_2 t^{\mu-1} H(t) \qquad \text{for a.e.~} t\in(0,+\infty) \, .
	\end{equation}
	Then, in view of the previous inequality and since $\mu c_2 t^{\mu-1} H(t) \geq 0$, we get
	\begin{equation} \label{P'0}
		\Phi'(t) \geq c_{3,\eps} t^{-\frac{\mu}{p}} \Phi(t) \qquad \text{for a.e.~} t>R_0 \, .
	\end{equation}
	We have $|\Omega_{s_0}|>0$ because $s_0<\esssup_M u$, so there exists $R_1>R_0$ such that $G(R_1)>0$. Let $R>R_1$ be given. By monotonicity of $G$ and since $c_2 t^\mu H(t)\geq 0$, we have $\Phi(t) \geq G(t)\geq G(R_1)>0$ for all $t\in[R_1,R]$. Since $[G(R_1),+\infty) \ni s \mapsto \log s$ is Lipschitz, the function $\log\Phi$ is absolutely continuous on $[R_1,R]$ with
	\[
		(\log\Phi)'(t) = \frac{\Phi'(t)}{\Phi(t)} \qquad \text{for a.e.~} t \in [R_1,R] \, .
	\]
	Thus, integrating \eqref{P'0} and using that $\Phi(R_1) \geq G(R_1)>0$ we get
	\begin{equation} \label{Pbase}
		\log\Phi(R) \geq \frac{c_{3,\eps}}{1-\frac{\mu}{p}} R^{1-\frac{\mu}{p}} + \log G(R_1) - \frac{c_{3,\eps}}{1-\frac{\mu}{p}} R_1^{1-\frac{\mu}{p}} \qquad \forall \, R>R_1 \, .
	\end{equation}
	Note that dividing both sides by $R^{1-\frac{\mu}{p}}$, letting $R\to+\infty$ and then $\eps\to0^+$ we would obtain
	\[
		\liminf_{R\to+\infty} \frac{1-\frac{\mu}{p}}{R^{1-\frac{\mu}{p}}} \log\Phi(R) \geq \lim_{\eps\to0^+} c_{3,\eps} = \frac{p(q-p+1)^{1/p'}\lambda^{1/p}}{(p-1)^{1/p'} k} \, ,
	\]
	which is (formally) weaker than \eqref{wp-0} since $\Phi(R)\geq G(R)$. To show that the same inequality holds with $\log G(R)$ in place of $\Phi(R)$, we proceed as follows. Let $R>R_1$ and $h>0$ be given. By inequality \eqref{Fg1} in Lemma \ref{lem_Fgamma} applied with the choice $\eps=\frac{1}{2}$ and with a cut-off function $0\leq\eta\in C^\infty_c(M)$ satisfying
	\[
	\begin{array}{rll}
		i) & \quad \eta\equiv 1 & \quad \text{on } \, B_R \, , \\[0.2cm]
		ii) & \quad \eta\equiv 0 & \quad \text{on } \, M\setminus B_{R+h} \, , \\[0.2cm]
		iii) & \quad 0 \leq \eta \leq 1 & \quad \text{on } \, B_{R+h} \setminus B_R \\[0.2cm]
		iv) & \quad |\nabla\eta| \leq \dfrac{2}{h} & \quad \text{on } \, M
	\end{array}
	\]
	we get
	\begin{equation} \label{GHh}
		\frac{k^{pp'}(p-1)^{p-1}4^p}{\gamma\min\{1,\gamma^{p-1}\}} G(R+h) \geq h^p H(R)
	\end{equation}
	and thus, choosing $h=R^{\mu/p}$,
	\begin{align*}
		\Phi(R) & = G(R) + c_2 R^\mu H(R) \\
		& \leq G(R) + \frac{c_2 k^{pp'}(p-1)^{p-1}4^p}{\gamma\min\{1,\gamma^{p-1}\}} G(R+R^{\mu/p}) \\
		& \leq \left(1+\frac{c_2 k^{pp'}(p-1)^{p-1}4^p}{\gamma\min\{1,\gamma^{p-1}\}}\right) G(R+R^{\mu/p}) =: C_2 G(R+R^{\mu/p})
	\end{align*}
	where in the last inequality we used monotonicity of $G$. Then, from \eqref{Pbase} we get
	\begin{equation} \label{GRR1}
		\log G(R+R^{\mu/p}) \geq \frac{c_{3,\eps}}{1-\frac{\mu}{p}} R^{1-\frac{\mu}{p}} + \log \frac{G(R_1)}{C_2} - \frac{c_{3,\eps}}{1-\frac{\mu}{p}} R_1^{1-\frac{\mu}{p}} \qquad \forall \, R>R_1 \, .
	\end{equation}
	Dividing both sides by $(R+R^{\mu/p})^{1-\frac{\mu}{p}}$ and then letting $R\to+\infty$ we get
	\[
		\liminf_{R\to+\infty} \frac{\log G(R+R^{\mu/p})}{(R+R^{\mu/p})^{1-\frac{\mu}{p}}} \geq \lim_{R\to+\infty} \frac{c_{3,\eps}}{1-\frac{\mu}{p}} \left( \frac{R}{R+R^{\mu/p}} \right)^{1-\frac{\mu}{p}} = \frac{c_{3,\eps}}{1-\frac{\mu}{p}}
	\]
	that is,
	\[
		\lim_{R\to+\infty} \frac{1-\frac{\mu}{p}}{R^{1-\frac{\mu}{p}}} \log G(R) \geq c_{3,\eps}
	\]
	and letting $\eps\to0^+$ we obtain \eqref{wp-0}.	
	
	\medskip
	
	\noindent \textbf{Case $\mu=p$.} Assume now that $\mu=p$. We first prove \eqref{wp-1}, and then \eqref{wp-1b} in the assumption that its LHS is well defined.
	
	\medskip
	
	\noindent \textbf{Proof of \eqref{wp-1}.} Choosing
	\[
		\sigma(t) = c_4 t^{-1/p'}
	\]
	for a suitable constant $c_4=c_{4,\eps}$ to be suitably selected later, from \eqref{GH'} we get
	\begin{equation} \label{GH'1}
		G'(t) + \frac{p-1}{c_4^{pp'}} t^p H'(t) \geq \frac{p(\lambda-\eps)}{c_4^p t} \left(G(t) + \frac{\gamma}{(\lambda-\eps)k^{p'}} t^p H(t)\right)
	\end{equation}
	for a.e.~$t>R_0$. In analogy with the previous case, we aim at using this to deduce an inequality of the form
	\begin{equation} \label{Phi'p}
		\Phi'(t) \geq c_5 t^{-1} \Phi(t) \qquad \text{for a.e.~} \, t > R_0
	\end{equation}
	with
	\begin{equation} \label{Phip}
		\Phi(t) = G(t) + c_6 t^p H(t)
	\end{equation}
	for suitable constants $c_5=c_{5,\eps}$ and $c_6=c_{6,\eps}$. Computing $\Phi'$ and rearranging terms we see that the desired inequality takes the form
	\begin{equation} \label{GH'2}
	\begin{split}
		G'(t) + c_6 t^p H'(t) & \geq c_5t^{-1} (G(t) + c_6 t^p H(t)) - p c_6 t^{p-1} H(t) \\
		& = c_5 t^{-1} \left( G(t) + c_6 \left(1-\frac{p}{c_5}\right) t^p H(t) \right)
	\end{split}
	\end{equation}
	so we want to choose $c_4$, $c_5$ and $c_6$ matching the following relations:
	\[
		\frac{p-1}{c_4^{pp'}} = c_6 \, , \qquad \frac{p(\lambda-\eps)}{c_4^p} = c_5 \, , \qquad c_6\left(1-\frac{p}{c_5}\right) = \frac{\gamma}{(\lambda-\eps)k^{p'}} \, .
	\]
	Expressing everything in terms of $c_5$ this amounts to
	\begin{equation} \label{c46}
		c_4 = \frac{p^{1/p}(\lambda-\eps)^{1/p}}{c_5^{1/p}} \, , \qquad c_6 = \frac{(p-1)c_5^{p'}}{p^{p'}(\lambda-\eps)^{p'}} \, ,
	\end{equation}
	\[
		\frac{\gamma}{(\lambda-\eps) k^{p'}} = \frac{c_6}{c_5}(c_5-p) = \frac{(p-1)c_5^{p'-1}(c_5-p)}{p^{p'}(\lambda-\eps)^{p'}} \, .
	\]
	that is, raising everything to the power $1/p'$ in the last relation, we choose $c_5=c_{5,\eps}$ as the unique value in $(p,+\infty)$ satisfying
	\[
		c_5^{1/p}(c_5-p)^{1/p'} = \frac{p\gamma^{1/p'}(\lambda-\eps)^{1/p}}{(p-1)^{1/p'}k} \,  (\,=c_{3,\eps}\,)
	\]
	and then we let $c_4$ and $c_6$ be defined accordingly by \eqref{c46}. Summarizing, for there choices of $c_4$, $c_5$ and $c_6$ we have that \eqref{GH'1} and \eqref{GH'2} coincide, and each of them is equivalent to \eqref{Phi'p} for $\Phi$ defined as in \eqref{Phip}. Then choosing $R_1>R_0$ such that $G(R_1)>0$ and reasoning as in the previous case we see that
	\[
		\log\Phi(R) \geq c_{5,\eps} \log R + \log G(R_1) - c_{5,\eps} \log R_1 \qquad \forall \, R>R_1
	\]
	and then by applying \eqref{GHh} with $h=R$ we obtain
	\[
		\log G(2R) \geq c_{5,\eps} \log R + \log G(R_1) - c_{5,\eps} \log R_1 - \log C_2 \qquad \forall \, R>R_1 \, .
	\]
	Dividing both sides by $\log(2R)$ and using that $\log(2R)\sim\log R$ as $R\to+\infty$ we get (after relabeling)
	\[
		\liminf_{R\to+\infty} \frac{\log G(R)}{\log R} \geq c_{5,\eps}
	\]
	and then letting $\eps\to0$ we get \eqref{wp-1}.
	
	\medskip
	
	\noindent \textbf{Proof of \eqref{wp-1b}.} Assume that
	\[
		\ell := \lim_{R\to+\infty} \frac{1}{\log R} \log\int_{B_R} (u-s_0)_+^q = \lim_{R\to+\infty} \frac{\log G(R)}{\log R}
	\]
	exists. From \eqref{wp-1} we already know that $\ell\geq C_1>p$. If $\ell=+\infty$ then \eqref{wp-1b} is trivially satisfied, so let us assume that $\ell<+\infty$. Let $\eps>0$ be as above and small enough so that $\ell-\eps > p$. Then there exists $R_2>R_0$ such that
	\begin{equation} \label{Gbil}
		R^p < R^{\ell-\eps} < G(R) < R^{\ell+\eps} \qquad \forall \, R>R_2 \, .
	\end{equation}
	We recall, from the discussion preceding the treatment of case $\mu<p$, that for each $t>R_2$ such that $G'(t)$ and $H'(t)$ exist we have \eqref{GH'0}, that is,
	\[
		\frac{\sigma^p}{p} G'(t) + \frac{\sigma^{-p'}}{p'} H'(t) \geq \int_{B_t} V w^q + \gamma k^{-p'} H(t)
	\]
	for any $\sigma>0$. Using the co-area formula twice together with \eqref{b1R0b} we get
	\begin{align*}
		\int_{B_t} V w^q \geq \int_{B_t\setminus B_{R_2}} V w^q & = \int_{R_2}^t \left( \int_{\partial B_s} V w^q \, \di\mathcal H^{m-1} \right) \, \di s \\
		& \geq \int_{R_2}^t \frac{\lambda-\eps}{s^p} \left( \int_{\partial B_s} w^q \, \di\mathcal H^{m-1} \right) \, \di s \\
		& = \int_{R_2}^t \frac{\lambda-\eps}{s^p} G'(s) \, \di s
	\end{align*}
	where $m=\dim M$ and $\mathcal H$ is the Hausdoff measure induced by the Riemannian structure. Substituting into the above inequality and multiplying both sides by $p\sigma^{-p}t^{-p}$ we get
	\[
		\frac{G'(t)}{t^p} + \frac{p-1}{\sigma^{pp'}} \frac{H'(t)}{t^p} \geq \frac{(\lambda-\eps)p}{\sigma^p t^p} \left[ \int_{R_2}^t \frac{G'(s)}{s^p} \, \di s + \frac{\gamma}{(\lambda-\eps)k^{-p'}} H(t) \right]
	\]
	and then choosing
	\begin{align*}
		c_1 = c_{1,\eps} & = (p-1)^{\frac{1}{pp'}} (\lambda-\eps)^{\frac{1}{pp'}} \gamma^{-\frac{1}{pp'}} k^{1/p} \\
		c_2 = c_{2,\eps} & = \frac{p-1}{c_1^{pp'}} \equiv \gamma (\lambda-\eps)^{-1} k^{-p'} \\
		c_3 = c_{3,\eps} & = \frac{p(\lambda-\eps)}{c_1^p} \equiv \frac{p\gamma^{1/p'}(\lambda-\eps)^{1/p}}{(p-1)^{1/p'}k} \\
		\sigma = \sigma(t) & = c_1 t^{-1/p'}
	\end{align*}
	this yields
	\begin{equation} \label{Ps'0}
		\frac{G'(t)}{t^p} + c_2 H(t) \geq \frac{c_3}{t} \left[ \int_{R_2}^t \frac{G'(s)}{s^p} \di s + c_2 H(t) \right] \qquad \text{for a.e.~} \, t > R_2 \, .
	\end{equation}
	Let $\Psi : (R_2,+\infty) \to [0,+\infty)$ be defined by
	\[
		\Psi(t) = \int_{R_2}^t \frac{G'(s)}{s^p} \, \di s + c_2 H(t) \, .
	\]
	The function $\Psi$ is absolutely continuous on each compact interval contained in $(R_2,+\infty)$ and inequality \eqref{Ps'0} can be restated as
	\begin{equation} \label{Ps'}
		\Psi'(t) \geq \frac{c_{3,\eps}}{t} \Psi(t) \qquad \text{for a.e.~} t>R_2 \, .
	\end{equation}
	Reasoning as in the previous cases, since $\Psi\not\equiv0$ we reach the conclusion
	\begin{equation} \label{liPs}
		\liminf_{R\to+\infty} \frac{\log\Psi(R)}{\log R} \geq c_{3,\eps} \, .
	\end{equation}
	We now use this to deduce \eqref{wp-1b}. Let $R>R_2$ be given. Applying \eqref{GHh} with $h=R$, integrating by parts and then using \eqref{Gbil} twice we get
	\begin{align*}
		\Psi(R) & \leq \int_{R_2}^R \frac{G'(s)}{s^p} \, \di s + C_2 \frac{G(2R)}{R^p} \\
		& = \frac{G(R)}{R^p} - \frac{G(R_2)}{R_2^p} + p \int_{R_2}^R \frac{G(s)}{s^{p+1}} \, \di s + C_2 \frac{G(2R)}{R^p} \\
		& \leq \frac{G(R)}{R^p} - \frac{G(R_2)}{R_2^p} + p \int_{R_2}^R s^{\ell+\eps-p-1} \, \di s + C_2 \frac{G(2R)}{R^p} \\
		& = \frac{G(R)}{R^p} - \frac{G(R_2)}{R_2^p} + \frac{p R^{\ell+\eps-p}}{\ell+\eps-p} - \frac{p R_2^{\ell+\eps-p}}{\ell+\eps-p} + C_2 \frac{G(2R)}{R^p} \\
		& \leq \frac{G(R)}{R^p} + \frac{p R^{2\eps}}{\ell+\eps-p} \frac{G(R)}{R^p} + C_2 \frac{G(2R)}{R^p} - \frac{G(R_2)}{R_2^p} - \frac{p R_2^{\ell+\eps-p}}{\ell+\eps-p} \, .
	\end{align*}
	Since $G$ is non-decreasing, we have $G(R)\leq G(2R)$ and then
	\[
		\Psi(R) \leq \left( \frac{p}{\ell+\eps-p} + (1+C_2) R^{-2\eps} \right) R^{-p+2\eps} G(2R) + O(1)
	\]
	as $R\to+\infty$. By \eqref{Gbil} we see that $R^{-p+2\eps} G(2R) > 2^p R^{2\eps} \to +\infty$, so
	\[
		\log\left[ \left( \frac{p}{\ell+\eps-p} + (1+C_2) R^{-2\eps} \right) R^{-p+2\eps} G(2R) + O(1) \right] \sim \log(R^{-p+2\eps} G(2R))
	\]
	as $R\to+\infty$, and then
	\begin{align*}
		\liminf_{R\to+\infty} \frac{\log\Psi(R)}{\log R} & \leq \liminf_{R\to+\infty} \frac{\log(R^{-p+2\eps} G(2R))}{\log R} \\
		& = -p+2\eps + \liminf_{R\to+\infty} \frac{\log(G(2R))}{\log R}
	\end{align*}
	and then, using that $\log R\sim \log(2R)$, after relabeling we get
	\[
		\liminf_{R\to+\infty} \frac{\log\Psi(R)}{\log R} \leq -p+2\eps + \lim_{R\to+\infty} \frac{\log G(R)}{\log R} \, .
	\]
	Substituting this into \eqref{liPs} yields
	\[
		\lim_{R\to+\infty} \frac{\log G(R)}{R} \geq c_{3,\eps} + p - 2\eps
	\]
	and then letting $\eps\to0^+$ we finally obtain \eqref{wp-1b}.
\end{proof}

\begin{remark}
	As a byproduct of the previous proof (namely, inequality \eqref{GRR1} above), we showed that if $u\in W^{1,p}_\loc(M)$ satisfies
	\[
		Lu \geq V u^{p-1} \qquad \text{on } \, \Omega_{s_0} = \{u>s_0\}
	\]
	with $V : M \to (0,+\infty)$ continuous and matching \eqref{Vlm1} for some $\lambda>0$ and $\mu\in[0,p]$, then for each $\eps\in(0,\lambda)$ and $R_0>0$ large enough (so that \eqref{b1R0b}-\eqref{b12R0} are satisfied) and for each $R_1>R_0$ such that
	\[
		I_1 := \int_{B_{R_1}} (u-s_0)_+^q > 0
	\]
	we have
	\begin{equation} \label{eqap}
		\log\int_{B_{R+R^{\mu/p}}} (u-s_0)_+^q \geq C_{0,\eps} \int_{R_1}^R t^{-\mu/p} \, \di t + \log\frac{I_1}{C_2} \qquad \forall \, R>R_1
	\end{equation}
	where
	\[
		C_{0,\eps} = \frac{p(q-p+1)^{1/p'}}{(p-1)^{1/p'}} \frac{(\lambda-\eps)^{1/p}}{k} \, , \qquad C_{2,\eps} = 1+\frac{k^p(p-1)^p 4^p}{(\lambda-\eps)\min\{1,\gamma^{p-1}\}}
	\]
	do not depend on $u$. Inequality \eqref{eqap} only involves the integrals of $w=(u-s_0)_+^q$ on geodesic balls, so it would still hold for functions $u\in L^q_\loc(M)$ that can be approximated pointwise and in $L^q$ norm on balls $B$ of arbitrary large radii by Sobolev functions $\tilde u\in W^{1,p}_\loc(B)$ satisfying
	\[
		L\tilde u \geq V|\tilde u|^{p-2}\tilde u \qquad \text{on } \, B \, .
	\]
	For instance, when $L=\Delta$ is the Laplace-Beltrami operator and $V\equiv 1$, a nontrivial result concerning local smooth monotone approximation of distributional $L^1_\loc$ subsolutions of $\Delta u = u$ (namely, Theorem D in \cite{bm22}) allows to extend the estimate
	\[
		\liminf_{R\to+\infty} \frac{1}{R} \int_{B_R} u_+^q \geq 2\sqrt{q-1}
	\] to distributional and not everywhere negative $L^1_\loc$ subsolutions of $\Delta u = u$.
\end{remark}

The following examples are aimed at showing the sharpness of the constant appearing in \eqref{wp-0} and \eqref{wp-1b}. Let $M$ be a model surface, that is, a complete Riemannian manifold diffeomorphic to $\R^2$ and radially symmetric around some point $o\in M$ so that in global polar coordinates $(r,\theta)$ centered at $o$ the metric takes the form
\[
	\metric = \di r^2 + g(r)^2 \di \theta^2
\]
for a smooth $g:(0,+\infty) \to (0,+\infty)$ satisfying $g'(0^+)=1$ and $g^{(2k)}(0^+)=0$ for each $k\in\{0\}\cup\N$. Let $v:[0,+\infty) \to \R$ be smooth and such that
\[
	v^{(k)}(0) = 0 \quad \forall \, k\in\N \qquad \text{and} \qquad v'(t)>0 \quad \forall \, t > 0 \, .
\]
Then $u:=v\circ r \in C^\infty(M)$, $|\nabla u|\neq 0$ on $M\setminus\{o\}$ and for any $p>1$ we have
\begin{equation} \label{Dpmod}
	\Delta_p u = \left[(p-1)(v')^{p-2} v'' + \frac{g'}{g} (v')^{p-1}\right] \circ r \qquad \text{on } \, M\setminus\{o\} \, .
\end{equation}

\textbf{Case $\mu\in[0,p)$.} Let $p>1$ and $\mu\in[0,p)$ be given. Consider $a,c\in\R$ satisfying
\begin{equation} \label{acex}
	c>0 \, , \qquad (p-1)c+a>0
\end{equation}
and set
\[
	\beta := 1-\frac{\mu}{p} \in (0,1] \, .
\]
Choose $g$ and $v$ satisfying the above requirements and such that
\[
	g(t) = \begin{cases}
		t & \quad \text{for } \, 0 < t \leq 1/2 \\
		\exp(at^\beta) & \quad \text{for } \, t \geq 1
	\end{cases}
\]
and
\[
	v(t) = \exp(ct^\beta) \qquad \text{for } \, t \geq 1 \, .
\]
By \eqref{Dpmod} we have
\begin{equation} \label{Dpex}
	\Delta_p u = V u^{p-1} \qquad \text{on } \, \Omega := M \setminus \overline{B_1}
\end{equation}
where
\begin{equation} \label{V0-ex}
	V = \left( (p-1)\left(1+\frac{\beta-1}{c\beta r^\beta}\right) c + a \right) \frac{\beta^p c^{p-1}}{r^\mu} \, .
\end{equation}
Let $s_0>e^c$. Since $v$ is non-decreasing, the set $\Omega_{s_0} := \{u>s_0\}$ coincides with $M\setminus\overline{B_{t_0}}$, where $t_0 = [(\log s_0)/c]^{1/\beta} > 1$, so in particular $\Omega_{s_0}\subseteq\Omega$. Also, for any $q>p-1$ we have
\[
	\int_{B_R} (u-s_0)_+^q = \int_{t_0}^R g(s) (v(s)-s_0)^q \, \di s \sim \int_{t_0}^R \exp((a+qc)s^\beta) \, \di s \qquad \text{as } \, R \to +\infty
\]
where the asymptotic equivalence between the integrals holds because
\[
	g(s)(v(s)-s_0)^q \sim \exp((a+qc)s^\beta) \to +\infty \qquad \text{as } \, s\to+\infty \, .
\]
(Recall that $a+qc>a+(p-1)c>0$ due to our assumptions on $a$ and $c$.) Integrating by parts yields
\begin{align*}
	\int_{t_0}^R \exp((a+qc)s^\beta) \, \di s & = \int_{t_0}^R \frac{\frac{\di}{\di s}\exp((a+qc)s^\beta)}{(a+qc)\beta s^{\beta-1}} \, \di s \\
	& = \frac{1}{(a+qc)\beta} \left( \frac{\exp((a+qc)R^\beta)}{R^{\beta-1}} - \frac{\exp((a+qc)t_0^\beta)}{t_0^{\beta-1}} \right) \\
	& \phantom{=\;} - \frac{1-\beta}{(a+qc)\beta} \int_{t_0}^R s^{-\beta} \exp((a+qc)s^\beta) \, \di s \\
	& \geq \frac{1}{(a+qc)\beta} \left( \frac{\exp((a+qc)R^\beta)}{R^{\beta-1}} - \frac{\exp((a+qc)t_0^\beta)}{t_0^{\beta-1}} \right) \\
	& \phantom{=\;} - \frac{(1-\beta)t_0^{-\beta}}{(a+qc)\beta} \int_{t_0}^R \exp((a+qc)s^\beta) \, \di s
\end{align*}
hence, rearranging terms and using that $\beta\in(0,1]$, we get
\[
	\frac{\exp((a+qc)R^\beta)}{a_1 R^{\beta-1}} + O(1) \geq \int_{t_0}^R \exp((a+qc)s^\beta) \, \di s \geq \frac{\exp((a+qc)R^\beta)}{a_2 \beta R^{\beta-1}} + O(1)
\]
for $R\to+\infty$, with
\[
	a_1 = (a+qc)\beta \, , \qquad a_2 = (a+qc)\beta + (1-\beta)t_0^{-\beta} \, .
\]
Passing to logarithms, we obtain
\[
	\log\int_{B_R} (u-s_0)_+^q \sim \log\int_{t_0}^R \exp((a+qc)s^\beta) \, \di s \sim (a+qc) R^\beta
\]
as $R\to+\infty$, that is, multiplying both sides by $\beta R^{-\beta}$ and recalling that $\beta=1-\frac{\mu}{p}$,
\[
	\lim_{R\to+\infty} \frac{1-\frac{\mu}{p}}{R^{1-\frac{\mu}{p}}} \log\int_{B_R} (u-s_0)_+^q = (a+qc)\beta \, .
\]
On the other hand, from \eqref{V0-ex} we clearly have
\begin{equation} \label{Vlac}
	\lim_{x\to\infty} r(x)^\mu V(x) = \lambda \qquad \text{with} \qquad \lambda = \beta^p c^{p-1}((p-1)c+a) \, .
\end{equation}
Since the $p$-Laplacian is weakly-$p$-coercive with coercivity constant $k=1$, to prove that estimate \eqref{wp-0} is sharp it is enough to show that for any $p$ and $q>p-1$ there exist $a$ and $c$ satisfying \eqref{acex} and such that
\begin{equation} \label{acpq}
	a+qc = \frac{p(q-p+1)^{1/p'}}{(p-1)^{1/p'}} c^{1/p'} ((p-1)c+a)^{1/p} \, .
\end{equation}
This can be done by picking any $a$ and $c>0$ such that
\[
	(p-1)a = (q-p(p-1))c
\]
since this would yield
\[
	a+qc = p((p-1)c+a) = \frac{p(q-p+1)}{p-1}c > 0
\]
and then
\begin{align*}
	a+qc = (a+qc)^{1/p'}(a+qc)^{1/p} & = \left(\frac{p(q-p+1)}{p-1}c\right)^{1/p'} \left(p((p-1)c+a)\right)^{1/p} \\
	& = \frac{p(q-p+1)^{1/p'}}{(p-1)^{1/p'}} c^{1/p'} ((p-1)c+a)^{1/p}
\end{align*}
as desired. For instance, a feasible choice for $a$ and $c$ would be the following:
\begin{equation}
	\left\{
		\begin{array}{cccl}
			a=-1 & \text{ and } & c = \dfrac{p-1}{p(p-1)-q} & \qquad \text{if } \, p-1 < q < p(p-1) \\[0.4cm]
			a=0 & \text{ and } & c = 1 & \qquad \text{if } \, q = p(p-1) \\[0.2cm]
			a=1 & \text{ and } & c = \dfrac{p-1}{q-p(p-1)} & \qquad \text{if } \, q > p(p-1) \, .
		\end{array}
	\right.
\end{equation}

\medskip

\textbf{Case $\mu=p$.} Let $p>1$ be given, consider $a,c\in\R$ satisfying \eqref{acex} and choose $g$ and $v$ satisfying the general requirements and such that
\[
	g(t) = \begin{cases}
		t & \qquad \text{for } \, 0 < t \leq 1/2 \\
		t^{a+p-1} & \qquad \text{for } \, t \geq 1
	\end{cases}
\]
and
\[
	v(t) = t^c \qquad \text{for } \, t \geq 1 \, .
\]
By \eqref{Dpmod} we have
\[
	\Delta_p u = V u^{p-1} \qquad \text{on } \, \Omega = M\setminus\overline{B_1}
\]
with
\[
	V = \frac{c^{p-1}((p-1)c+a)}{r^p} \, .
\]
Let $s_0>1$ be given. Then $\Omega_{s_0} := \{u>s_0\}$ is contained in $\{u>1\} = M\setminus\overline{B_1}$ and for any $q>p-1$ we have
\[
	\log\int_{B_R} (u-s_0)_+^q \sim \log\int_{s_0}^R s^{a+p-1+qc} \, \di s \sim (a+p+qc) \log R \qquad \text{as } \, R \to +\infty
\]
that is,
\[
	\lim_{R\to+\infty} \frac{1}{\log R} \log\int_{B_R}(u-s_0)_+^q = (a+qc) + p
\]
and then again to prove sharpness of \eqref{wp-1b} we need to show that for any $p>1$ and $q>p-1$ we can choose $a$ and $c$ satisfying \eqref{acex} and
\[
	a+qc = \frac{p(q-p+1)^{1/p'} c^{1/p'}((p-1)c+a)^{1/p}}{(p-1)^{1/p'}} \, ,
\]
but this is precisely what we did in the previous case.

\section{The case $Lu\geq 0$} \label{sec4}

In this section we are concerned with lower bounds on the growth of functions $u$ satisfying the differential inequality $Lu\geq 0$ on a non-empty superlevel set. The main result of this section is Theorem \ref{parq_0} below, corresponding to Theorem \ref{thm-in5} from the Introduction. The starting point in this case is again Lemma \ref{lem_Fm}. For ease of the reader we point out that in this case it takes the following form.

\begin{lemma} \label{lem_FF}
	Let $M$ be a Riemannian manifold, $p\in(1,+\infty)$ and $L$ a weakly-$p$-coercive operator as in \eqref{Ldef}. Let $u\in W^{1,p}_\loc(M)$ satisfy
	\begin{equation}
		Lu \geq 0 \qquad \text{on } \, \Omega_{s_0} := \{ x \in M : u(x) > s_0 \}
	\end{equation}
	for some $s_0\in\R$. Then for any $0\leq\eta\in C^\infty_c(M)$ and for any non-negative, non-decreasing, piecewise $C^1$ function on $(0,+\infty)$ we have
	\begin{equation} \label{s_alt}
		\int_{\Omega_{s_0}} F(w) |A_u||\nabla\eta|  \geq \int_{\Omega_{s_0}} \eta F'(w) |A_u|^{p'}
	\end{equation}
	where $w:=(u-s_0)_+$ and $A_u:=A(x,u,\nabla u)$.
\end{lemma}

The main tool to prove Theorem \ref{parq_0} is the next proposition.

\begin{proposition} \label{prop_RS}
	Let $M$ be a complete, non-compact Riemannian manifold, $p\in(1,+\infty)$ and $L$ a weakly-$p$-coercive operator as in \eqref{Ldef}. Let $u\in W^{1,p}_\loc(M)$ satisfy
	\begin{equation}
		Lu \geq 0 \qquad \text{on } \, \Omega_{s_0} := \{ x \in M : u(x) > s_0 \}
	\end{equation}
	for some $s_0\in\R$.
	
	\textbf{(a)} For any $q>p-1$ and for any $x_0\in M$ and $0<r<R$
	\begin{equation} \label{pgp}
		\int_{B_r(x_0)\cap\Omega_{s_0}} w^{q-p} |A_u|^{p'} \leq \frac{(p-1)^{p-1}}{\min\{1,\gamma^p\}} \left( \int_r^R \left( \int_{\partial B_s(x_0)} w^q \right)^{1/(1-p)} \di s \right)^{1-p}
	\end{equation}
	where $w:=(u-s_0)_+$, $A_u:=A(x,u,\nabla u)$ and $\gamma:=q-p+1$.
	
	\textbf{(b)} If $u_+\in L^\infty_\loc(M)$ and $F$ is a non-negative, piecewise $C^1$ function on $(0,+\infty)$ such that $F'>0$ everywhere on $(0,+\infty)$, then
	\begin{equation} \label{pfF}
		\int_{B_r(x_0)\cap\Omega_{s_0}} F'(w)|A_u|^{p'} \leq (p-1)^{p-1} \left( \int_r^R \left( \int_{\Omega_{s_0}\cap\partial B_s(x_0)} \frac{[F(w)]^p}{[F'(w)]^{p-1}} \right)^{1/(1-p)} \di s \right)^{1-p}
	\end{equation}
	for every $x_0\in M$ and $0<r<R$, with $w$ and $A_u$ as above.
\end{proposition}

\begin{remark}
	We remark that the exponents $1-p$ and $1/(1-p)$ appearing on the RHS's of \eqref{pgp} and \eqref{pfF} are negative. With the agreement that $0^a = +\infty$ and $(+\infty)^a = 0$ for any $a\in(-\infty,0)$, the inequalities make sense also in case one or more of the integrals on the RHS's are either vanishing or diverging.
\end{remark}

\begin{proof}
	Let $w$ and $A_u$ be as in the statement. We first prove (b), since the proof of (a) relies on the same idea coupled with suitable approximation arguments.
	
	\medskip
	
	\textbf{Proof of (b).} Suppose that $u_+\in L^\infty_\loc(M)$ and let $F$ be as in the statement. The function $F$ satisfies all the requirements in Lemma \ref{lem_FF} and therefore
	\begin{equation} \label{Ff0}
		\int_{\Omega_{s_0}} F(w)|A_u||\nabla\eta| \geq \int_{\Omega_{s_0}} \eta F'(w) |A_u|^{p'}
	\end{equation}
	for any $0\leq\eta\in C^\infty_c(M)$. Note that both integrals are finite since $F(w),F'(w)\in L^\infty(\Omega_{s_0})$ and $|A_u|\mathbf{1}_{\Omega_{s_0}}\in L^{p'}_\loc(M)$. Applying H\"older inequality with conjugate exponents $p$ and $p'$ as in \eqref{FFgYo} we further obtain
	\begin{equation} \label{Ho0}
		\left( \int_{\Omega_{s_0}} F'(w)|A_u|^{p'}|\nabla\eta| \right)^{1/p'} \left( \int_{\Omega_{s_0}} \frac{[F(w)]^p}{[F'(w)]^{p-1}}|\nabla\eta| \right)^{1/p} \geq \int_{\Omega_{s_0}} \eta F'(w) |A_u|^{p'}
	\end{equation}
	where the middle integral is again finite since $[F(w)]^p/[F'(w)]^{p-1} \in L^\infty(\Omega_{s_0})$. Let $x_0\in M$ be fixed and let us write $B_s$ for the geodesic ball $B_s(x_0)$, for any $s>0$. Let $G,H:(0,+\infty) \to [0,+\infty)$ be defined by
	\begin{equation} \label{GHdef}
		G(s) := \int_{\Omega_{s_0}\cap B_s} F'(w) |A_u|^{p'} \, , \qquad H(s) := \int_{\Omega_{s_0}\cap B_s} \frac{[F(w)]^p}{[F'(w)]^{p-1}} \, .
	\end{equation}
	Since $F'(w)|A_u|^{p'}\mathbf{1}_{\Omega_{s_0}}\in L^1_\loc(M)$ and $[F(w)]^p/[F'(w)]^{p-1} \mathbf{1}_{\Omega_{s_0}} \in L^\infty(M) \subseteq L^1_\loc(M)$, the functions $G$ and $H$ are well defined, non-decreasing and absolutely continuous on any compact interval contained in $(0,+\infty)$. In particular, they are differentiable a.e.~on $(0,+\infty)$. Let $s>0$ be a value for which $G'(s)$ and $H'(s)$ both exist. For any $\eps>0$ choose $\eta_\eps\in C^\infty_c(M)$ satisfying
	\[
	\begin{array}{rll}
		i) & \quad \eta_\eps\equiv 1 & \quad \text{on } \, B_s \, , \\[0.2cm]
		ii) & \quad \eta_\eps\equiv 0 & \quad \text{on } \, M\setminus B_{s+\eps} \, , \\[0.2cm]
		iii) & \quad 0 \leq \eta_\eps \leq 1 & \quad \text{on } \, B_{s+\eps} \setminus B_r \\[0.2cm]
		iv) & \quad |\nabla\eta_\eps| \leq \dfrac{1}{\eps} + 1 & \quad \text{on } \, M \, .
	\end{array}
	\]
	Then
	\begin{align*}
		\int_{\Omega_{s_0}} F(w)|A_u||\nabla\eta_\eps| & \leq \left(\frac{1}{\eps}+1\right)\int_{\Omega_{s_0}\cap B_{s+\eps}\setminus B_s} F(w)|A_u| \leq (1+\eps) \frac{G(s+\eps)-G(s)}{\eps}
	\end{align*}
	and passing to limits as $\eps\to0^+$ we get
	\[
		\limsup_{\eps\to0^+} \int_{\Omega_{s_0}} F(w)|A_u||\nabla\eta_\eps| \leq G'(s) \in [0,+\infty) \, .
	\]
	Similarly, we obtain
	\[
		\limsup_{\eps\to0^+} \int_{\Omega_{s_0}} \frac{[F(w)]^p}{[F'(w)]^{p-1}} |\nabla\eta_\eps| \leq H'(s)
	\]
	and by dominated convergence theorem we also have
	\[
		\lim_{\eps\to0^+} \int_{\Omega_{s_0}} \eta_\eps F'(w) |A_u|^{p'} = G(s) \, .
	\]
	Then by \eqref{Ho0} we deduce
	\begin{equation} \label{HG'}
		[H'(s)]^{p'/p} G'(s) \geq [G(s)]^{p'} \qquad \text{for a.e.~} \, s>0 \, .
	\end{equation}
	Moreover, by the co-area formula we have
	\begin{equation} \label{H'h}
		H'(s) = \int_{\Omega_{s_0}\cap\partial B_s} \frac{[F(w)]^p}{[F'(w)]^{p-1}} =: \varphi(s) \qquad \text{for a.e.~} \, s>0 \, .
	\end{equation}
	Let $0<r<R$ be given. If $G(r)=0$ then \eqref{pfF} is trivially satisfied. If $G(r)>0$ then by monotonicity of $G$ we have that $G(s)\geq G(r)$ for all $s\in[r,R]$. Since $G'(s)$ is finite for a.e.~$s\in[r,R]$, from \eqref{HG'} and \eqref{H'h} we infer that $\varphi(s)>0$ for a.e.~$s\in[r,R]$ and then
	\begin{equation} \label{G'phi}
		\frac{G'(s)}{[G(s)]^{p'}} \geq [\varphi(s)]^{-p'/p} \qquad \text{for a.e.~} s \in [r,R] \, .
	\end{equation}
	Since $G(s)\geq G(r)>0$ for all $s\in[r,R]$ and $[G(r),+\infty) \ni t\mapsto t^{1/(1-p)}$ is Lipschitz, the function $G^{1/(1-p)} \equiv G^{1-p'}$ is absolutely continuous on $[r,R]$ with
	\[
		\frac{\di}{\di s} [G(s)]^{1-p'} = \frac{1}{1-p} \frac{G'(s)}{[G(s)]^{p'}} \qquad \text{for a.e.~} s \in [r,R] \, .
	\]
	Thus, integrating \eqref{G'phi} we get (noting that $p'/p=1/(p-1)$)
	\[
		(p-1)\left[ G(r)^{-1/(p-1)} - G(R)^{-1/(p-1)} \right] = \int_r^R \frac{G'(s)}{G(s)^{p'}} \, \di s \geq \int_r^R [\varphi(s)]^{1/(1-p)} \, \di s \, .
	\]
	Discarding the term containing $G(R)$ and raising everything to $1-p$ we get
	\[
		G(r) \leq (p-1)^{p-1}\left( \int_r^R [\varphi(s)]^{1/(1-p)} \, \di s \right)^{1-p}
	\]
	that is, \eqref{pfF}.
	
	\medskip
	
	\textbf{Proof of (a).} We observe that the argument developed above can be applied straightforwardly, without the assumption $u_+\in L^\infty_\loc(M)$, as long as we consider a piecewise $C^1$ function $F:(0,+\infty) \to (0,+\infty)$ with $F'>0$ such that
	\begin{equation} \label{faF}
		F'(w)|A_u|^{p'} \mathbf{1}_{\Omega_{s_0}} \in L^1_\loc(M) \, , \qquad \frac{[F(w)]^p}{[F'(w)]^{p-1}} \mathbf{1}_{\Omega_{s_0}} \in L^1_\loc(M) \, .
	\end{equation}
	Indeed, if the conditions in \eqref{faF} are satisfied then all the integrals appearing in \eqref{Ff0} and \eqref{Ho0} are finite and the functions $G$ and $H$ defined as in \eqref{GHdef} are again finite-valued, non-decreasing and absolutely continuous on every compact interval contained in $(0,+\infty)$.
	
	\medskip
	
	\textbf{Case $q\geq p$.} Set $\gamma := q-p+1 \geq 1$. For any $h>0$ define $F_h$ by
	\begin{equation} \label{Fh}
		F_h(s) := \begin{cases}
			\dfrac{s^\gamma}{\gamma} & \quad \text{if } \, 0 < s < h \\[0.4cm]
			\dfrac{h^\gamma}{\gamma} + (s-h) h^{\gamma-1} & \quad \text{if } \, s \geq h \, .
		\end{cases}
	\end{equation}
	Note that $F_h$ is positive and $C^1$ on $(0,+\infty)$ with 
	\begin{equation} \label{fh}
		F_h'(s) = \begin{cases}
			s^{\gamma-1} & \quad \text{if } \, 0 < s < h \\
			h^{\gamma-1} & \quad \text{if } \, s \geq h \, .
		\end{cases}
	\end{equation}
	We have $F_h'>0$ everywhere on $(0,+\infty)$ and $F_h'(w) \in L^\infty(\Omega_{s_0})$, therefore also $F_h'(w)|A_u|^{p'}\mathbf{1}_{\Omega_{s_0}} \in L^1_\loc(M)$, due to \eqref{wpC4} and $u\in W^{1,p}_\loc(M)$. Moreover,
	\begin{equation} \label{Fhw}
		\begin{split}
		\frac{[F_h(w)]^p}{[F_h'(w)]^{p-1}}\mathbf{1}_{\Omega_{s_0}} & = \frac{w^{\gamma+p-1}}{\gamma^p} \mathbf{1}_{\{0<w<h\}} + h^{\gamma-1} \left( w - \frac{\gamma-1}{\gamma} h \right)^p \mathbf{1}_{\{w\geq h\}} \\
		& \leq \frac{w^{\gamma+p-1}}{\gamma^p} \mathbf{1}_{\{0<w<h\}} + h^{\gamma-1} w^p \mathbf{1}_{\{w\geq h\}}
		\end{split}
	\end{equation}
	so in particular
	\[
		\frac{[F_h(w)]^p}{[F_h'(w)]^{p-1}}\mathbf{1}_{\Omega_{s_0}} \leq h^{\gamma-1} w^p \in L^1_\loc(M)
	\]
	since $\gamma\geq 1$ and $u\in W^{1,p}_\loc(M)$. Hence, conditions \eqref{faF} are satisfied for $F=F_h$ and we can repeat the argument in the proof of \textbf{(a)} up to obtaining
	\[
		[\varphi_h(s)]^{p'/p} G_h'(s) \geq [G_h(s)]^{p'} \qquad \text{for a.e.~} \, s > 0
	\]
	with
	\[
		G_h(s) = \int_{\Omega_{s_0}\cap B_s} F_h'(w)|A_u|^{p'} \, , \qquad \varphi_h(s) = \int_{\Omega_{s_0}\cap\partial B_s} \frac{[F_h(w)]^p}{[F_h'(w)]^{p-1}} \, .
	\]
	From \eqref{Fhw} and recalling that $\gamma=q-p+1$ we also have
	\[
		\frac{[F_h(w)]^p}{[f_h(w)]^{p-1}}\mathbf{1}_{\Omega_{s_0}} \leq w^q \qquad \text{on } \, M
	\]
	hence
	\[
		\varphi_h(s) \leq \varphi(s) := \int_{\partial B_s} w^q \qquad \forall \, s>0 \, .
	\]
	Reasoning again as in the proof of \textbf{(a)} we deduce that either $G_h(r)=0$ or
	\[
		\begin{cases}
			G_h(s)\geq G_h(r)>0 & \quad \forall s\in[r,R] \\[0.2cm]
			\dfrac{G_h'(s)}{[G_h(s)]^{p'}} \geq [\varphi_h(s)]^{1/(1-p)} \geq [\varphi(s)]^{1/(1-p)} & \quad \text{for a.e.~} s\in[r,R] \, .
		\end{cases}
	\]
	In any case we get
	\[
		\int_{B_r} \min\{w,h\}^{\gamma-1}|A_u|^{p'} = G_h(r) \leq (p-1)^{p-1} \left( \int_r^R [\varphi(s)]^{1/(1-p)} \di s \right)^{1-p}
	\]
	and the conclusion follows by the monotone convergence theorem letting $h\to+\infty$.
	
	\medskip
	
	\textbf{Case $p-1<q<p$.} Set $\gamma := q-p+1$ as in the previous case and note that now $\gamma\in(0,1)$. For any $h>0$ let $F_h$ be defined as in \eqref{Fh}. We note that $F_h$ is positive and $C^1$ on $(0,+\infty)$ in this case too, with $F_h'>0$ everywhere on $(0,+\infty)$. Then from Lemma \ref{lem_FF} we get
	\begin{equation} \label{faF'}
		\int_{\Omega_{s_0}} F_h(w)|A_u||\nabla\eta| \geq \int_{\Omega_{s_0}} \eta F_h'(w)|A_u|^{p'} \qquad \forall \, 0 \leq \eta \in C^\infty_c(M) \, .
	\end{equation}
	From the expression \eqref{Fh} we see that $F_h(w) \leq C_{h,\gamma}(1+w)$, hence $F_h(w)|A_u|\mathbf{1}_{\Omega_{s_0}}\in L^1_\loc(M)$ by H\"older inequality. By \eqref{faF'} this also yields
	\[
		F_h'(w)|A_u|^{p'} \mathbf{1}_{\Omega_{s_0}} \in L^1_\loc(M) \, .
	\]
	On the other hand, we have
	\begin{equation} \label{Fh01}
		\begin{split}
			\frac{[F_h(w)]^p}{[F_h'(w)]^{p-1}} \mathbf{1}_{\Omega_{s_0}} & = \frac{w^{\gamma+p-1}}{\gamma^p} \mathbf{1}_{\{0<w<h\}} + h^{\gamma-1} \left( w-h+\frac{h}{\gamma} \right)^p \mathbf{1}_{\{w\geq h\}} \\
			& \leq \frac{w^{\gamma+p-1}}{\gamma^p} \mathbf{1}_{\{0<w<h\}} + h^{\gamma-1} \left( \frac{w-h}{\gamma}+\frac{h}{\gamma} \right)^p \mathbf{1}_{\{w\geq h\}} \\
			& = \frac{w^{\gamma+p-1}}{\gamma^p} \mathbf{1}_{\{0<w<h\}} + \frac{h^{\gamma-1}w^p}{\gamma^p} \mathbf{1}_{\{w\geq h\}}
		\end{split}
	\end{equation}
	where the inequality in the middle holds because $w-h<(w-h)/\gamma$ on $\{w>h\}$, since $0<\gamma<1$ in this case. From this estimate we get
	\[
		\frac{[F_h(w)^p]}{[F_h'(w)]^{p-1}}\mathbf{1}_{\Omega_{s_0}} \leq \max\left\{\frac{h^q}{\gamma^p},\frac{h^{\gamma-1}}{\gamma^p} w^p\right\} \in L^1_\loc(M) \, .
	\]
	Hence, both conditions in \eqref{faF} are satisfied. Setting again
	\[
		G_h(s) = \int_{\Omega_{s_0}\cap B_s} F_h'(w)|A_u|^{p'} \, , \qquad \varphi_h(s) = \int_{\Omega_{s_0}\cap\partial B_s} \frac{[F_h(w)]^p}{[F_h'(w)]^{p-1}}
	\]
	we can repeat once more the general argument to get that either $G_h(r)=0$ or
	\[
		\begin{cases}
			G_h(s)\geq G_h(r)>0 & \quad \forall s\in[r,R] \\[0.2cm]
			\dfrac{G_h'(s)}{[G_h(s)]^{p'}} \geq [\varphi_h(s)]^{1/(1-p)} & \quad \text{for a.e.~} s\in[r,R]
		\end{cases}
	\]
	and in any case we get
	\begin{equation} \label{Gh01}
		\int_{B_r} F_h'(w) |A_u|^{p'} = G_h(r) \leq (p-1)^{p-1} \left( \int_r^R [\varphi_h(s)]^{1/(1-p)} \di s \right)^{1-p} .
	\end{equation}
	We now let $h\to+\infty$ in both sides of \eqref{Gh01}. By Fatou's lemma we have
	\begin{equation} \label{Gh01a}
		\liminf_{h\to+\infty} \int_{\Omega_{s_0}\cap B_r} F_h'(w) |A_u|^{p'} \geq \int_{\Omega_{s_0}\cap B_r} w^{\gamma-1} |A_u|^{p'} \equiv \int_{\Omega_{s_0}\cap B_r} w^{q-p} |A_u|^{p'} \, .
	\end{equation}
	Concerning the RHS of \eqref{Gh01}, we aim at showing that
	\begin{equation} \label{phsl}
		\lim_{h\to+\infty} \int_r^R [\varphi_h(s)]^{1/(1-p)} \di s = \int_r^R [\varphi(s)]^{1/(1-p)} \di s
	\end{equation}
	with
	\[
		\varphi(s) := \frac{1}{\gamma^p}\int_{\partial B_s} w^q \, .
	\]
	From \eqref{Fh01} and recalling that $\gamma+p-1=q$ we have
	\begin{equation} \label{phihs}
		0 \leq \varphi_h(s) - \frac{1}{\gamma^p} \int_{\partial B_s\cap\{w<h\}} w^q \leq \frac{h^{\gamma-1}}{\gamma^p} \int_{\partial B_s\cap\{w\geq h\}} w^p \, .
	\end{equation}
	Since $w\in W^{1,p}_\loc(M)$, for a.e.~$s\in[r,R]$ we have $w\in L^p(\partial B_s)$ by the co-area formula. Then, using the monotone convergence theorem on the first integral in \eqref{phihs} together with the fact that $h^{\gamma-1} \to 0$ as $h\to+\infty$ (due to $\gamma<1$) we get
	\begin{equation} \label{phihs_pc}
		\lim_{h\to+\infty} \varphi_h(s) = \frac{1}{\gamma^p} \int_{\partial B_s} w^q = \varphi(s) \qquad \text{for a.e.~} s \in [r,R] \, .
	\end{equation}
	If $\varphi^{1/(1-p)}\not\in L^1([r,R])$, then by \eqref{phihs_pc} and Fatou's lemma we have
	\[
		\liminf_{h\to+\infty} \int_r^R \varphi_h^{1/(1-p)} \geq \int_r^R \varphi^{1/(1-p)} = +\infty
	\]
	so \eqref{phsl} holds with both sides equalling $+\infty$. Suppose, instead, that $\varphi^{1/(1-p)}\in L^1([r,R])$. From the first line in \eqref{Fh01} we also deduce the reversed estimate
	\begin{align*}
		\frac{[F_h(w)]^p}{[F_h'(w)]^{p-1}} \mathbf{1}_{\Omega_{s_0}} & \geq \frac{w^{\gamma+p-1}}{\gamma^p} \mathbf{1}_{\{0<w<h\}} + h^{\gamma-1} w^p \mathbf{1}_{\{w\geq h\}} \\
		& \geq \frac{w^{\gamma+p-1}}{\gamma^p} \mathbf{1}_{\{0<w<h\}} + w^{\gamma-1} w^p \mathbf{1}_{\{w\geq h\}} \\
		& \geq w^{\gamma+p-1} = w^q
	\end{align*}
	where in the second inequality we exploited again the fact that $0<\gamma<1$. Then, for every $h>0$ we also have $\varphi_h \geq \gamma^p \, \varphi$ and therefore
	\[
		\varphi_h^{1/(1-p)} \leq \gamma^{-p/(p-1)} \varphi^{1/(1-p)} \qquad \text{on } \, [r,R] \, .
	\]
	Hence, if $\varphi^{1/(1-p)} \in L^1([r,R])$ then \eqref{phsl} follows by the dominated convergence theorem. In any case, from the continuity of $[0,+\infty] \ni t \mapsto t^{1-p} \in [0,+\infty]$ with the agreement that $0^{1-p} = +\infty$ and $(+\infty)^{1-p} = 0$ we get
	\begin{equation} \label{Gh01b}
		\lim_{h\to+\infty} \left( \int_r^R [\varphi_h(s)]^{1/(1-p)} \di s \right)^{1-p} = \left( \int_r^R [\varphi(s)]^{1/(1-p)} \di s \right)^{1-p} \, .
	\end{equation}
	By \eqref{Gh01}, \eqref{Gh01a} and \eqref{Gh01b} we obtain the desired conclusion.	
\end{proof}

From Proposition \ref{prop_RS} we easily deduce the following lemma.

\begin{lemma} \label{bd_int>0}
	Let $M$ be a complete, non-compact Riemannian manifold, $p\in(1,+\infty)$ and $L$ a weakly-$p$-coercive operator as in \eqref{Ldef}. Let $u\in W^{1,p}_\loc(M)$ satisfy
	\begin{equation}
		Lu \geq 0 \qquad \text{on } \, \Omega_{s_0} := \{ x \in M : u(x) > s_0 \}
	\end{equation}
	for some $s_0\in\R$ and also suppose that
	\begin{equation} \label{AE0>}
		A(x,u,\nabla u) \neq 0 \qquad \text{on a set $E_0\subseteq\Omega_{s_0}$ of positive measure.}
	\end{equation}
	Then there exists $r_0\geq0$ such that for any $q>p-1$
	\begin{equation}
		\int_r^R \left( \int_{\partial B_s} (u-s_0)_+^q \right)^{1/(1-p)} \di s < +\infty \qquad \forall \, r_0 < r < R < +\infty \, .
	\end{equation}
	In particular,
	\begin{equation} \label{Hfin}
		\mathcal{H}^{m-1}(\Omega_{s_0}\cap\partial B_r) > 0 \qquad \text{for a.e.~} r>r_0
	\end{equation}
	where $\mathcal{H}^{m-1}$ denotes the $(m-1)$-dimensional Hausdorff measure. Moreover, if $u_+\in L^\infty_\loc(M)$ then also
	\begin{equation} \label{Hint_fin}
		0 < \int_r^R \left( \mathcal{H}^{m-1}(\Omega_{s_0}\cap\partial B_s) \right)^{1/(1-p)} \di s < +\infty \qquad \forall \, r_0 < r < R < +\infty \, .
	\end{equation}
\end{lemma}

\begin{proof}
	Choose $r_0\geq 0$ such that $|B_r\cap E_0|>0$ for every $r>r_0$, where $E_0$ is as in \eqref{AE0>}. Then, for every $r>r_0$
	\[
		\int_{B_r\cap\Omega_{s_0}} w^{q-p}|A_u|^{p'} \geq \int_{B_r\cap E_0} w^{q-p}|A_u|^{p'} > 0
	\]
	and then applying Proposition \ref{prop_RS}.(a) we see that the RHS of \eqref{pgp} must be strictly positive for any $R>r$, that is (since $1-p < 0$),
	\[
		\int_r^R \left( \int_{\partial B_s} w^q \right)^{1/(1-p)} \di s < +\infty \qquad \forall \, R>r \, .
	\]
	In particular, $\left(\int_{\partial B_s} w^q\right)^{1/(1-p)}$ must be finite for a.e.~$s>r$, hence for a.e.~$s>r_0$ by arbitrariness of $r>r_0$, and therefore it must be $\int_{\partial B_s} w^q>0$ for a.e.~$s>r_0$, yielding \eqref{Hfin}. If $u_+\in L^\infty_\loc(M)$, to prove \eqref{Hint_fin} we start from the two-sided estimate
	\[
		\mathcal{H}^{m-1}(\partial B_s) \geq \mathcal{H}^{m-1}(\Omega_{s_0}\cap\partial B_s) \geq \frac{1}{(1+\esssup_{B_R} w)^p} \int_{\Omega_{s_0}\cap\partial B_s} (1+w)^p \, ,
	\]
	holding for each $s>r_0$, from which we deduce
	\begin{align*}
		\left(\mathcal{H}^{m-1}(\partial B_s)\right)^{1/(1-p)} & \leq \left(\mathcal{H}^{m-1}(\Omega_{s_0}\cap\partial B_s)\right)^{1/(1-p)} \\
		& \leq (1+\esssup_{B_R} w)^{p/(p-1)} \left(\int_{\Omega_{s_0}\cap\partial B_s} (1+w)^p\right)^{1/(1-p)} .
	\end{align*}
	The function $v(r) := \mathcal{H}^{m-1}(\partial B_r)$ satisfies
	\begin{equation} \label{vBMR}
		v(r) > 0 \quad \text{for } \, r > 0 \qquad \text{and} \qquad v, \frac{1}{v} \in L^\infty_\loc((0,+\infty))
	\end{equation}
	see Proposition 1.6 in \cite{bmr13}, so we have
	\[
		\int_r^R \left( \mathcal{H}^{m-1}(\partial B_s) \right)^{1/(1-p)} \di s > 0 \qquad \forall \, 0 < r < R
	\]
	and by Proposition \ref{prop_RS}.(b) applied with the choice $f\equiv 1$ and $F(s) = 1 + s$ we get
	\[
		\int_r^R \left( \int_{\partial B_s} (1+w)^p \right)^{1/(1-p)} \di s < +\infty \qquad \forall \, r_0 < r < R \, .
	\]
	Putting together all inequalities above we obtain \eqref{Hint_fin}.
\end{proof}

We are now ready for the proof of the main result of this section.

\begin{theorem} \label{parq_0}
	Let $M$ be a complete, non-compact Riemannian manifold, $p\in(1,+\infty)$ and $L$ a weakly-$p$-coercive operator as in \eqref{Ldef}. Let $u\in W^{1,p}_\loc(M)$ satisfy
	\[
		Lu \geq 0 \qquad \text{on } \, \Omega_{s_0} := \{ x\in M : u(x) > s_0 \}
	\]
	for some $s_0\in\R$ and suppose that for some $x_0\in M$ and $q>p-1$ it holds
	\begin{equation} \label{nL1_q}
		\lim_{R\to+\infty} \int_r^R \left( \int_{\partial B_s(x_0)} (u-s_0)_+^q \right)^{-\frac{1}{p-1}} \di s = +\infty \qquad \forall \, r > 0 \, .
	\end{equation}
	Then $A(x,u,\nabla u) = 0$ a.e.~on $\Omega_{s_0}$. Thus, if $A$ satisfies the structural condition
	\begin{equation} \label{pC1}
		A(x,s,\xi) = 0 \qquad \text{if and only if} \qquad \xi = 0 \, .
	\end{equation}
	then either $u\equiv c$ a.e.~on $M$ for some constant $c>s_0$, or $u\leq s_0$ a.e.~on $M$.
\end{theorem}

\begin{remark}
	Condition \eqref{nL1_q} can be stated, more briefly, as
	\[
		\left( \int_{\partial B_s} (u-s_0)_+^q \right)^{-\frac{1}{p-1}} \not \in L^1(+\infty)
	\]
	with this notation meaning that the function $\varphi : (0,+\infty) \to [0,+\infty]$ given by
	\[
		\varphi(s) = \left( \int_{\partial B_s} (u-s_0)_+^q \right)^{-\frac{1}{p-1}} \qquad \forall \, s > 0
	\]
	is not in $L^1((r,+\infty))$ for any $r>0$. The previous Lemma \ref{bd_int>0} implies that this is a meaningful condition, since in general only two cases are possible:
	\begin{itemize}
		\item[(i)] $\varphi=+\infty$ a.e.~on $(0,+\infty)$, and then $\Omega_{s_0}$ has zero measure while condition \eqref{nL1_q} is obviously satisfied, or
		\item[(ii)] there exists $r_0\geq0$ such that $\varphi<+\infty$ a.e.~on $(r_0,+\infty)$ and $\varphi \in L^1((r,R))$ for any $r_0<r<R<+\infty$, so that \eqref{nL1_q} is satisfied if and only if $\varphi$ is not integrable in a neighbourhood of $+\infty$.
	\end{itemize}
	Concerning case (ii), note that in general $\varphi$ may be integrable at $+\infty$ and still satisfy $\varphi=+\infty$ on $(0,r_0)$ for some $r_0>0$ (for instance, on $\R^n$ this may happen if $u$ satisfies $u\leq s_0$ on $B_{r_0}$ and $u(x)\geq |x|^a$ as $x\to\infty$ for some $a>(p-n)/q$), so the clause ``$\forall\,r>0$'' in \eqref{nL1_q} cannot in general be replaced by ``for some $r>0$''.
\end{remark}

\begin{proof}[Proof of Theorem \ref{parq_0}]
	Suppose, by contradiction, that $A_u := A(x,u,\nabla u)$ is non-zero on a set $E_0\subseteq\Omega_{s_0}$ of positive measure. Then reasoning as in the proof of Lemma \ref{bd_int>0} we see that there exists $r>0$ such that
	\[
		\int_{\Omega_{s_0}\cap B_r} w^{q-p}|A_u|^{p'} > 0
	\]
	and by Proposition \ref{prop_RS} this implies that
	\[
		\int_r^R \left( \int_{\partial B_s} (u-s_0)_+^q \right)^{1/(1-p)} \di s \leq \left( \frac{\min\{1,\gamma^p\}}{(p-1)^{p-1}}\int_{\Omega_{s_0}\cap B_r} w^{q-p}|A_u|^{p'}\right)^{1/(1-p)}
	\]
	for all $R>r$, with $\gamma = q-p+1$. Since the RHS of this inequality is finite, letting $R\to+\infty$ in the LHS we reach the desired contradiction. So, we conclude that $A(x,u,\nabla u)=0$ a.e.~on $\Omega_{s_0}$.
	
	If $A$ satisfies the non-degeneracy condition \eqref{pC1} then we further deduce that $\nabla u=0$ a.e.~on $\Omega_{s_0}$, and since the function $w := (u-s_0)_+ \in W^{1,p}_\loc(M)$ has weak gradient $\nabla w = \mathbf{1}_{\Omega_{s_0}}\nabla u$ this yields $\nabla w \equiv 0$ a.e.~on $M$. By connectedness of $M$ this implies that $w = a$ a.e.~on $M$ for some constant $a\geq 0$. If $a>0$ then $u = c := s_0+a$ a.e.~on $M$ (and $\Omega_{s_0}$ is of full measure), while if $a=0$ then $u\leq s_0$ a.e.~on $M$ (and $\Omega_{s_0}$ has zero measure).
\end{proof}

\medskip

As a consequence of Theorem \ref{parq_0} we have the following Liouville-type theorem.

\begin{corollary} \label{parq_1}
	Let $M$ be a complete, non-compact Riemannian manifold, $p\in(1,+\infty)$ and $L$ a weakly-$p$-coercive operator as in \eqref{Ldef}. Let $u\in W^{1,p}_\loc(M)$ satisfy
	\[
		Lu \geq 0 \qquad \text{on } \, \Omega_{s_0} := \{ x\in M : u(x) > s_0 \}
	\]
	for some $s_0\in\R$ and suppose that for some $x_0\in M$ and $q>p-1$ it holds
	\begin{equation} \label{nL1_qb}
		\lim_{R\to+\infty} \int_r^R \left( \frac{s}{\int_{B_s} (u-s_0)_+^q} \right)^{\frac{1}{p-1}} \di s = +\infty \qquad \forall \, r > 0 \, .
	\end{equation}
	Then $A(x,u,\nabla u) = 0$, and if $A$ satisfies the structural condition \eqref{pC1} then either $u\equiv c$ a.e.~on $M$ for some $c>s_0$ or $u\leq s_0$ a.e.~on $M$.
\end{corollary}

\begin{proof}
	The corollary is a direct consequence of Theorem \ref{parq_0} since \eqref{nL1_qb} implies \eqref{nL1_q}. For the details, see the proof of Proposition 1.3 in \cite{rs} (the parameter $\delta$ there corresponds to $p-1$ in our setting).
\end{proof}

\begin{remark} \label{parq_2}
	Note, in particular, that \eqref{nL1_qb} holds if
	\begin{equation}
		\int_{B_R} (u-s_0)_+^q = O(R^p) \qquad \text{as } \, R\to+\infty
	\end{equation}
	or even if, for some $n\in\N$
	\begin{equation}
		\int_{B_R} (u-s_0)_+^q = O(R^p g_n^{p-1}(R)) \qquad \text{as } \, R\to+\infty \, .
	\end{equation}
	where
	\[
		g_n(t) = (\log t)(\log\log t)\cdots(\underbrace{\log\log\cdots\log t}_{n\text{ iterations}}) \qquad \text{for } \, t >> 1 \, .
	\]
\end{remark}

\bibliographystyle{plain}

\begin{thebibliography}{99}
	\bibitem{bmr13} B. Bianchini, L. Mari, M. Rigoli, \textit{On some aspects of oscillation theory and geometery}, Mem. AMS, \textbf{225} (2013), 1056. vi+195 pp.
	
	\bibitem{bm22} A. Bisterzo, L. Marini, \textit{The $L^\infty$-positivity Preserving Property and Stochastic Completeness}, Potential Anal. (2022), Online first DOI: \href{https://doi.org/10.1007/s11118-022-10041-w}{\texttt{10.1007/s11118-022-10041-w}}
	
	\bibitem{cmr22} G. Colombo, L. Mari and M. Rigoli, \textit{Einstein-type structures, Besse's conjecture, and a uniqueness result for a $\varphi$-CPE metric in its conformal class}, J. Geom. Anal. \textbf{32} (2022), no. 11, Paper No. 267, 32 pp. MR4470304
	
	\bibitem{dam17} L. D'Ambrosio, E. Mitidieri, \textit{Quasilinear elliptic equations with critical potentials.} Adv. Nonlinear Anal. \textbf{6} (2017), no. 2, 147--164. MR3641630
	
	\bibitem{gt} D. Gilbarg and N. S. Trudinger, ``Elliptic partial differential equations of second order. Reprint of the 1998 edition''. Classics in Mathematics. Springer-Verlag, Berlin, 2001. MR1814364 
	
	\bibitem{le98} V. K. Le, \textit{On some equivalent properties of sub- and supersolutions in second order quasilinear elliptic equations}, Hiroshima Math. J \textbf{28} (1998), no. 2,  373--380. MR1637342
	
	\bibitem{rs} M. Rigoli and A. G. Setti, \textit{Liouville type theorems for $\varphi$-subharmonic functions}, Rev. Mat. Iberoamericana \textbf{17} (2001), no. 3, 471--520. MR1900893
	
	\bibitem{r18} F. Rindler, ``Calculus of Variations''. Universitext. Springer, Cham, 2018. MR3821514
\end{thebibliography}

\end{document}